\documentclass[a4paper, 10pt, notitlepage]{article}

\usepackage{amsthm, amsmath, amssymb, latexsym}
\usepackage{mathrsfs,cite}
\usepackage[multiple]{footmisc}
\usepackage{graphicx}
\usepackage[ruled,vlined]{algorithm2e}

\theoremstyle{plain}
\newtheorem{theorem}{Theorem}[section]
\newtheorem{lemma}{Lemma}[section]
\newtheorem{corollary}{Corollary}[section]
\newtheorem{proposition}{Proposition}[section]

\theoremstyle{definition}
\newtheorem{definition}{Definition}[section]
\newtheorem{example}{Example}[section]

\theoremstyle{remark}
\newtheorem{remark}{Remark}[section]

\DeclareMathOperator{\co}{co}
\DeclareMathOperator{\cl}{cl}
\DeclareMathOperator{\dom}{dom}
\DeclareMathOperator{\interior}{int}
\DeclareMathOperator{\dist}{dist}
\DeclareMathOperator{\core}{core}
\DeclareMathOperator{\trace}{Tr}

\author{M.V. Dolgopolik\footnote{Institute for Problems in Mechanical Engineering, Saint Petersburg, Russia}
\footnote{This work was performed in IPME RAS and supported by the Russian Science Foundation 
(Grant No. 20-71-10032).}}
\title{Hypodifferentials of nonsmooth convex functions and their applications to nonsmooth convex optimization}

\begin{document}

\maketitle

\begin{abstract}
A hypodifferential is a compact family of affine mappings that defines a local max-type approximation of a nonsmooth
convex function. We present a general theory of hypodifferentials of nonsmooth convex functions defined on a Banach
space. In particular, we provide complete characterizations of hypodifferentiability and hypodifferentials of nonsmooth
convex functions, derive calculus rules for hypodifferentials, and study the Lipschitz continuity/Lipschitz
approximation property of hypodifferentials that can be viewed as a natural extension of the Lipschitz continuity of 
the gradient to the general nonsmooth setting. As an application of our theoretical results, we study the rate of
convergence of several versions of the method of hypodifferential descent for nonsmooth convex optimization and present
an accelerated version of this method having the faster rater of convergence $\mathcal{O}(1/k^2)$.
\end{abstract}

\section{Introduction}

Although the idea of approximating objective function with the use of a family of affine mappings has been successfully
used in the context of minimax optimization problems (see, e.g. \cite[Section~2.4]{Polak} and
\cite[Sections~2.3.1--2.3.3]{Nesterov_book}), in the general nonsmooth case only linear approximations are typically
used. For example, a subdifferential is always defined as a collection of some sort of local or global \textit{linear}
approximations/minorants of a nonsmooth function under consideration \cite{Penot,Thibault}. 

In the late 1980s, Demyanov \cite{Demyanov_InCollection_1988,Demyanov1988,Demyanov1989,DemyanovRubinov} extended 
the idea of approximating a nonsmooth function with the use of families of affine mappings to the case of general
nonsmooth functions by introducing the concept of \textit{codifferential}. Since then the theory of codifferentials has
found a number of applications in optimization, calculus of variations, nonsmooth mechanics, etc.
(see \cite{Quasidifferentiability_book,QuasidiffMechanics,Dolgopolik_COCV,Dolgopolik_COCV_2} and the references
therein).

In the convex case (and some other cases, e.g. in the cases of minimax and composite problems), codifferentials 
are reduced to the so-called \textit{hypodifferentials}, which can be defined as a compact family of affine mappings
that provides a local max-type approximation of a nonsmooth function. A numerical method based on the use of
hypodifferentials, called \textit{the method of hypodifferential descent} \cite{DemyanovRubinov}, has been applied to
various problems of the calculus of variations \cite{DemyanovTamasyan2011,DemyanovTamasyan2014} and optimal control
problems \cite{Fominyh,Fominyh_OptimLetters}. In \cite{Dolgopolik_HypodiffDescent}, it was shown that the method of
hypodifferential descent with Armijo's step size rule has the rate of convergence $\mathcal{O}(1/k)$ for general
nonsmooth convex functions, which exceeds the standard rate of convergence $\mathcal{O}(1/\sqrt{k})$ of subgradient
methods \cite{Nesterov_book}, although at the cost of significantly increased complexity of each iteration. The fact
that the method of hypodifferential descent has the same rate of convergence in the nonsmooth case as gradient descent
methods in the smooth case hints at the possibility of devising accelerated versions of this method that are similar to
accelerated gradient methods.

An accelerated version of the gradient descent method having the $\mathcal{O}(1/k^2)$ rate of convergence for smooth
convex functions with globally Lipschitz continuous gradient has been first developed by Nesterov \cite{Nesterov83}
(see also \cite{Nesterov_book}). Various improved versions of accelerated gradient descent methods have been studied,
e.g. in \cite{KimFessler2016,KimFessler2018}, while a constructive technique for devising accelerated methods for smooth
and nonsmooth convex minimization was proposed in \cite{DroriTaylor}. A general approach towards design and analysis
of accelerated gradient methods based on analysis of auxiliary ordinary differential equations and differential
equations solvers was recently developed in \cite{ShiDuJordanSu,LuoChen}.

As is well-known (see, e.g. \cite{Nesterov_book}), in the general nonsmooth case the rate of convergence of convex
minimization methods is significantly slower than in the smooth case and is equal to $\mathcal{O}(1/\sqrt{k})$. For a
fixed number of iterations, this rate of convergence is achieved by the celebrated Shor's subgradient method
\cite{Shor}, while when the number of iterations is not fixed it is achieved by many existing methods, such as
Nesterov's primal-dual subgradient methods \cite{Nesterov2009}, OSGA \cite{Neumaier}, a Kelley-like method from
\cite{DroriTeboulle} (see also restarted subgradient method \cite{YangLin}), etc. A significantly faster rate of
convergence can be achieved in the nonsmooth case with the use of some additional assumptions on the objective
function, such as H\"{o}lderian growth \cite{JohnstoneMoulin}, the validity of error bounds and the
Kurdyka-{\L}ojasiewicz inequality \cite{BolteNguyenPeypouquetSuter}, function growth condition measures \cite{FreundLu},
etc. A faster rate of convergence can also be achieved by exploiting a particular structure of a problem class under
consideration, as is done, e.g. in the context of Nesterov's smoothing \cite{Nesterov2005} and excessive gap
\cite{Nesterov2005_2} techniques and in the setting of composite convex optimization
\cite{Xu,Teboulle,TaylorHendrickxGlineur,GrapigliaNesterov}.

The main goal of this paper is twofold. Firstly, our aim is to develop a general theory of hypodifferentials of
nonsmooth convex functions defined on a Banach space. We provide complete characterizations of hypodifferentiability and
hypodifferentials of such functions in terms of subdifferentiability and subdifferentials in the sense of convex
analysis. We also present a detailed hypodifferential calculus and study the Lipschitz continuity/Lipschitz
approximation property of hypodifferential. In contrast to subdifferentials, hypodifferentials of many nonsmooth
convex functions arising in applications are continuous and often even Lipschitz continuous. Thus, the Lipschitz
continuity of a hypodifferential can be viewed as a natural extension of the property of the Lipschitz continuity of 
the gradient, which is widely used in smooth convex optimization, to the nonsmooth setting. We show that standard
operations on convex functions (including the pointwise maximum) preserve the Lipschitz continuity of hypodifferential
and, moreover, provide a general existence theorem for Lipschitz continuous hypodifferentials.

Secondly, as an application of our theoretical results, we aim at analysing the rate of convergence of various versions
of the method of hypodifferential descent in the nonsmooth convex case and presenting an accelerated version of this
method having a faster rate of convergence. Namely, we prove that the method of hypodifferential descent with constant
step size has the rate of convergence $\mathcal{O}(1/k)$ and show that under some additional assumptions a specific
choice of step size rule can ensure the linear convergence of the method. Following the standard approach to
accelerated methods developed by Nesterov \cite{Nesterov_book}, we also present an accelerated version of the method of 
hypodifferential descent for a class of nonsmooth convex functions that can be viewed as a natural extension of the
class of smooth convex functions having a globally Lipschitz continuous gradient to the nonsmooth case. We prove that
the accelerated method has the rate of convergence $\mathcal{O}(1/k^2)$ that coincides with the rate of convergence of
accelerated gradient methods. It must be pointed out that such rate of convergence is achieved by significantly
increasing the complexity of each iteration and each oracle call in comparison with classical subgradient/bundle
methods, which make the accelerated hypodifferential descent method inapplicable to large scale problems. Nonetheless,
this method can be viewed as a method that bridges the gap between nonsmooth and smooth convex optimization methods, at
least from the purely theoretical point of view.

Let us finally note that this paper is a continuation of the research started in \cite{Dolgopolik_HypodiffDescent}.
Some of the results presented below are completely new (e.g. the characterizations of hypodifferentiability and
hypodifferentials given in Theorem~\ref{thrm:HypodiffCharacterization}, some of the calculus rules and results on
Lipschitz property, the accelerated version of the method of hypodifferential descent, etc.), while others are
significantly revised and improved versions of the author's earlier results from
\cite{Dolgopolik_CodiffDescent,Dolgopolik_HypodiffDescent}.

The paper is organized as follows. Some necessary auxiliary results on the Pompeiu-Hausdorff distance are collected in
Section~\ref{sect:Preliminaries}. The hypodifferential calculus for nonsmooth convex functions is developed in 
Section~\ref{sect:HypodiffTheory}. Subsection~\ref{subsect:HypodiffMain} contains main definitions connected to
hypodifferentials, as well as the theorem providing convenient characterizations of hypodifferentiability and
hypodifferentials. Calculus rules for hypodifferentials are presented in Subsection~\ref{subsect:CalculusRules}, while
several particular examples are discussed in Subsection~\ref{subsect:Examples}.
Subsection~\ref{subsect:LipschitzProperty} is devoted to the analysis of Lipschitz continuity/Lipschitz approximation
property of hypodifferentials. Some applications of our theoretical results to nonsmooth convex optimization are
presented in Section~\ref{sect:Applications}. The rate of convergence of the method of hypodifferential descent with two
different step size rules is estimated in Subsection~\ref{subsect:MHD}, while an accelerated version of this method
having a faster rate of convergence is studied in Subsection~\ref{subsect:AcceleratedPHD}.

\section{Preliminaries}
\label{sect:Preliminaries}

Let us present some auxiliary results on the Pompeiu-Hausdorff metric and the continuity of set-valued maps with
respect to this metric that will be utilized throughout the article. It should be noted that all results presented
in this section are simple and well-known (see \cite{Thibault,RockafellarWets}), but scattered in the literature and
often presented either without proofs or only in the finite dimensional/compact-valued case. Therefore, for the sake of
completeness we provide short proofs of all these results.

Let $(M, d)$ be a metric space. Recall that the Pompeiu-Hausdorff distance between nonempty subsets $A, B \subseteq M$
is defined as
\[
  d_{PH}(A, B) = \sup\Big\{ \sup_{a \in A} \dist(a, B), \: \sup_{b \in B} \dist(b, A) \Big\},
\]
where $\dist(a, B) = \inf_{b \in B} d(a, b)$ is the distance between a point $a \in M$ and a set $B \subseteq M$. The
following simple result follows directly from the definition of the Pompeiu-Hausdorff distance.

\begin{proposition} \label{prp:PHdist}
Let $A, B \subseteq M$ be nonempty sets and $\varepsilon > 0$ be given. Then $d_{PH}(A, B) < \varepsilon$ if and only
if there exists $\eta \in (0, \varepsilon)$ for which the following two conditions hold true:
\begin{enumerate}
\item{for any $a \in A$ there exists $b \in B$ such that $d(a, b) \le \eta$;
}

\item{for any $b \in B$ there exists $a \in A$ such that $d(a, b) \le \eta$.
}
\end{enumerate}
\end{proposition}

Let $(X, d_X)$ be a metric space. Recall that a set-valued map $F \colon M \rightrightarrows X$ is called
\textit{continuous} in the Pompeiu-Hausdorff metric at a point $x \in M$, if for any $\varepsilon > 0$ there exists
$\delta > 0$ such that 
\[
  d_{PH}(F(x'), F(x)) < \varepsilon \quad \forall 
  x' \in B(x, \delta) := \big\{ y \in M \bigm| d(x, y) \le \delta \big\}.
\]
The map $F$ is called continuous in the Pompeiu-Hausdorff metric, if it is continuous at every point $x \in M$. Note
also that if one replaces $B(x, \delta)$ with a neighbourhood of $x$, then the definition above can be applied to the
case when $M$ is merely a topological space.

\begin{remark}
Below we will often call set-valued maps that are continuous in the Pompeiu-Hausdorff metric simply continuous, since we
will not consider any other notion of continuity for set-valued mappings in this paper. 
\end{remark}

Let us show that many standard operations on set-valued maps preserve continuity in the Pompeiu-Hausdorff metric.

\begin{proposition} \label{prp:SumContinuous}
Let $X$ be a normed space and $F, G \colon M \rightrightarrows X$ be continuous set-valued maps. Then the Minkowski sum
$H = F + G$ is continuous as well.
\end{proposition}

\begin{proof}
Fix any $x \in M$ and $\varepsilon > 0$. Since $F$ and $G$ are continuous, one can find $\delta > 0$ such that
\[
  d_{PH}(F(x'), F(x)) < \frac{\varepsilon}{2}, \quad d_{PH}(G(x'), G(x)) < \frac{\varepsilon}{2}
  \quad \forall x' \in B(x, \delta).
\]
Choose any $x' \in B(x, \delta)$ and $y \in H(x')$. By definition there exist $y_1 \in F(x')$ and $y_2 \in G(x')$ such
that $y = y_1 + y_2$. By Proposition~\ref{prp:PHdist} there exist $\eta < \varepsilon / 2$, $z_1 \in F(x)$, and 
$z_2 \in G(x)$ such that $\| y_i - z_i \| \le \eta$, $i \in \{ 1, 2 \}$. Define $z = z_1 + z_2$. Clearly, 
$z \in H(x)$ and $\| y - z \| \le 2 \eta < \varepsilon$. Thus, for any $y \in H(x')$ we have found $z \in H(x)$ such
that $\| y - z \| \le 2 \eta < \varepsilon$. Swapping $x$ and $x'$ we get that for any $z \in H(x)$ one can find 
$y \in H(x')$ such that $\| y - z \| \le 2 \eta < \varepsilon$. Therefore, by Proposition~\ref{prp:PHdist} one has
$d_{PH}(H(x'), H(x)) < \varepsilon$, which implies that $H$ is continuous.
\end{proof}

\begin{proposition} \label{prp:ProductContinuous}
Let $X$ be a normed space, $F \colon M \rightrightarrows X$ be a continuous set-valued map, and 
$g \colon M \to \mathbb{R}$ be a continuous function. Suppose also that $F(x)$ is a bounded set for any $x \in X$. Then 
the set-valued map $H(\cdot) = g(\cdot) F(\cdot)$ is continuous.
\end{proposition}

\begin{proof}
Fix any $x \in M$ and $\varepsilon > 0$. From the continuity of $g$ it follows that there exist $\theta > 0$ and 
$\delta_1 > 0$ such that $|g(x')| \le \theta$ for any $x' \in B(x, \delta_1)$. In turn, from the continuity of $F$ it
follows that there exists $\delta_2 > 0$ such that $d_{PH}(F(x'), F(x)) < \varepsilon / 2 \theta$ for any 
$x' \in B(x, \delta_2)$. Therefore by Proposition~\ref{prp:PHdist} there exists $\eta < \varepsilon$ such that for any 
$x' \in B(x, \delta_2)$ and $y' \in F(x')$ one can find $y \in F(x)$ for which $\| y' - y \| \le \eta / 2 \theta$. 
Finally, applying the continuity of $g$ again one gets that there exists $\delta_3 > 0$ such
that $|g(x') - g(x)| < \eta / 2C$ for any $x' \in B(x, \delta_3)$, where $C = \sup\{ \| y \| \mid y \in F(x) \}$.

Set $\delta = \min\{ \delta_1, \delta_2, \delta_3 \}$. Choose any $x' \in B(x, \delta)$ and $z' \in H(x')$. 
Then $z' = g(x') y'$ for some $y' \in F(x')$. As was noted above, one can find $y \in F(x)$ such that 
$\| y' - y \| \le \eta / 2 \theta$. Define $z = g(x) y$. Then $z \in H(x)$ and
\[
  \| z' - z \| = \| g(x') y' - g(x) y \| \le |g(x')| \| y' - y \| + |g(x') - g(x)| \| y \|
  \le \theta \frac{\eta}{2 \theta} + \frac{\eta}{2 C} C \le \eta.
\]
Thus, for any $z' \in H(x')$ we have found $z \in H(x)$ such that $\| z' - z \| \le \eta$. Arguing in a similar way one
can check that for any $z \in H(x)$ there exists $z' \in H(x')$ satisfying the same inequality. Hence by
Proposition~\ref{prp:PHdist} one has $d_{PH}(H(x'), H(x)) < \varepsilon$, which implies the required result.
\end{proof}

\begin{proposition} \label{prp:UnionContinuous} 
Let $T$ be a nonempty set and $F \colon M \times T \rightrightarrows X$ be a given set-valued map. Suppose that the maps
$F(\cdot, t)$ are continuous uniformly in $t$. Then set-valued map $G(\cdot) = \bigcup_{t \in T} F(\cdot, t)$ is
continuous. In particular, if $T$ is a compact topological space, then it is sufficient to assume that the set-valued
map $F$ is continuous.
\end{proposition}

\begin{proof}
Fix any $x \in M$ and $\varepsilon > 0$. From the uniform in $t$ continuity of $F(\cdot, t)$ it follows that one can
find $\delta > 0$ such that $d_{PH}(F(x', t), F(x, t)) < \varepsilon/2$ for any $x' \in B(x, \delta)$ and $t \in T$.

Choose any $x' \in B(x, \delta)$ and $z' \in G(x')$. Then by definition there exists $t' \in T$ such that 
$y' \in F(x', t')$. By Proposition~\ref{prp:PHdist} one can find $y \in F(x, t') \subseteq G(x)$ such that 
$\| y' - y \| < \varepsilon/2$. Similarly, if $y \in G(x)$, then $y \in F(x, t)$ for some $t \in T$ and
Proposition~\ref{prp:PHdist} implies that there exists $y' \in F(x', t) \subseteq G(x')$ such that 
$\| y' - y \| < \varepsilon/2$. Hence applying Proposition~\ref{prp:PHdist} one more time one obtains that
$d_{PH}(G(x'), G(x)) < \varepsilon$, which implies that $G$ is continuous.

It remains to note that if $T$ is a compact topological space and the set-valued map $F$ is continuous, then, as one can
readily check using the standard compactness argument, the maps $F(\cdot, t)$ are continuous uniformly in $t$.
\end{proof}

\begin{proposition} \label{prp:ConvexHullContinuous}
Let $X$ be a normed space and $F \colon M \rightrightarrows X$ be a continuous set-valued map. Then the set-valued map
$\co F$ is continuous as well, where ``$\co$'' stands for the convex hull.
\end{proposition}

\begin{proof}
Fix any $\varepsilon > 0$ and $x \in M$. Due to the continuity of $F$ there exists $\delta > 0$ such that 
$d_{PH}(F(x'), F(x)) < \varepsilon$ for any $x' \in B(x, \delta)$. Choose any $x' \in B(x, \delta)$ and 
$y' \in \co F(x')$. By the definition of convex hull there exist $n \in \mathbb{N}$, $y'_i \in F(x')$, and 
$\alpha_i \ge 0$, $i \in \{ 1, \ldots, n \}$, such that
\[
  y' = \sum_{i = 1}^n \alpha_i y'_i, \quad \sum_{i = 1}^n \alpha_i = 1.
\]
By Proposition~\ref{prp:PHdist} there exists $\eta < \varepsilon$ such that for any $i \in \{ 1, \ldots, n \}$ one can
find $y_i \in F(x)$ for which $\| y'_i - y_i \| \le \eta$. Define $y = \sum_{i = 1}^n \alpha_i y_i$. Then 
$y \in \co F(x)$ and $\| y' - y \| \le \eta$. Thus, for any $y' \in F(x')$ we found $y \in F(x)$ such that 
$\| y' - y \| \le \eta$. Repeating the same argument with $y'$ and $y$ swapped and then applying
Proposition~\ref{prp:PHdist} one obtains that $d_{PH}(\co F(x'), \co F(x)) < \varepsilon$, which implies that 
the set-valued map $\co F$ is continuous.
\end{proof}

\begin{proposition} \label{prp:ClosureContinuous}
Let $F \colon M \rightrightarrows X$ be a continuous set-valued map. Then the map $\cl F$ is continuous as well, where
$\cl$ is the closure operator.
\end{proposition}

\begin{proof}
The validity of this proposition follows from the triangle inequality for the Pompeiu-Hausdorff metric and the obvious
equality $d_{PH}(A, \cl A) = 0$ that holds true for any nonempty set $A \subset X$.
\end{proof}

Let $X$ be a real Banach space. The topological dual space of $X$ is denoted by $X^*$, and the canonical duality pairing
between $X$ and $X^*$ is denoted by $\langle \cdot, \cdot \rangle$, that is, $\langle x^*, x \rangle = x^*(x)$ for any
$x \in X$ and $x^* \in X^*$. The weak topology on $X$ is denoted by $\sigma(X, X^*)$, while the weak$^*$ topology on
$X^*$ is denoted by $\sigma(X^*, X)$. 

We endow the Cartesian product $X \times Y$ of normed spaces with the norm 
$\| (x, y) \| = \| x \|_X + \| y \|_Y$ for any $(x, y) \in X \times Y$. As is easily seen, the spaces
$(\mathbb{R} \times X)^*$ and $\mathbb{R} \times X^*$ are isomorphic and the corresponding isomorphism is continuous
both in the weak and weak${}^*$ topologies. In addition, the weak${}^*$ topology on $(\mathbb{R} \times X)^*$
corresponds to the the product topology $\tau_{\mathbb{R}} \times \sigma(X^*, X)$ on $\mathbb{R} \times X^*$, where
$\tau_{\mathbb{R}}$ is the standard topology on $\mathbb{R}$. For the sake of shortness we call this product topology
weak${}^*$. Let us also note that subsets of $\mathbb{R} \times X^*$ are weak${}^*$ compact if and only if they are norm
bounded and weak${}^*$ closed (see, e.g. \cite[Theorem~2.1]{Dolgopolik_CodiffCalc}). We denote the closure in the
weak${}^*$ topology by $\cl^*$.

\begin{proposition} \label{prp:WeakStarClosureContinuous}
Let $X$ be a reflexive Banach space, $Q \subseteq X$ be a nonempty set, and
$F \colon Q \rightrightarrows \mathbb{R} \times X^*$ be a set-valued map with bounded convex values. If $F$ is
continuous, then $\cl^* F$ is continuous as well.
\end{proposition}

\begin{proof}
One can readily check that the reflexivity of $X$ implies that the space $\mathbb{R} \times X$ is reflexive as well.
Hence taking into account the fact that for convex sets the closure in the strong topology coincides with the closure
in the weak topology one can conclude that for any convex subset $K \subset (\mathbb{R} \times X)^*$ its closure in the
strong topology $\cl K$ coincides with $\cl^* K$. Therefore $\cl F(x) = \cl^* F(x)$ for any $x \in Q$,
which by Proposition~\ref{prp:ClosureContinuous} implies that $\cl^* F$ is continuous
\end{proof}

\section{Hypodifferential calculus for nonsmooth convex functions}
\label{sect:HypodiffTheory}

In this section we present a detailed theory of \textit{hypodifferentials} of nonsmooth convex functions defined on 
a real Banach space. Our aim is to develop a calculus of such approximations and show that hypodifferentials possess
many useful properties in the general nonsmooth case (such as continuity) that subdifferentials do not have.

\subsection{Hypodifferentials of convex functions}
\label{subsect:HypodiffMain}

Let $X$ be a real Banach space and $f \colon X \to \mathbb{R} \cup \{ + \infty \}$ be a closed convex function. Denote
by $\dom f = \{ x \in X \mid f(x) \in \mathbb{R} \}$ the effective domain of $f$.

\begin{definition} \label{def:Hypodifferential}
The function $f$ is called \textit{hypodifferentiable} at a point $x \in \dom f$, if there exists a weak${}^*$ compact
set $\underline{d} f(x) \subset \mathbb{R} \times X^*$ such that
\begin{equation} \label{eq:HypodiffDef}
\begin{split}
  \lim_{\alpha \to +0} \frac{1}{\alpha} \Big| f(x + \alpha \Delta x) - f(x) 
  - \max_{(a, x^*) \in \underline{d} f(x)} (a + \alpha \langle x^*, \Delta x \rangle) \Big| &= 0,
  \\
  \max_{(a, x^*) \in \underline{d} f(x)} a &= 0
\end{split}
\end{equation}
for any $\Delta x \in X$. The set $\underline{d} f(x)$ is called a \textit{hypodifferential} of $f$ at $x$.
\end{definition}

\begin{remark}
It should be noted the maximum in $\max_{(a, x^*) \in \underline{d} f(x)} (a + \langle x^*, y \rangle)$ is attained 
for any $y \in X$, since the set $\underline{d} f(x) \subset \mathbb{R} \times X^*$ is weak${}^*$ compact.
\end{remark}

\begin{remark} \label{rmrk:HypodiffDef}
Let us point out how the definition of hypodifferential transforms under various assumptions on the space $X$:
\begin{enumerate}
\item{If $X$ is a reflexive Banach space, then a hypodifferential of $f$ at $x \in \dom f$ is a norm closed and bounded
subset $\underline{d} f(x) \subseteq \mathbb{R} \times X^*$ satisfying equalities~\eqref{eq:HypodiffDef}.
}

\item{If $X$ is a Hilbert space, then taking into account the Riesz representation theorem it is natural to
define a hypodifferential of $f$ as a weakly compact (or, equivalently, norm closed and bounded) set 
$\underline{d} f(x) \subset \mathbb{R} \times X$ such that for any $\Delta x \in X$ one has 
\begin{equation} \label{eq:HypodiffDefHilbertCase}
\begin{split}
  \lim_{\alpha \to +0} \frac{1}{\alpha} \Big| f(x + \alpha \Delta x) - f(x) 
  - \max_{(a, v) \in \underline{d} f(x)} (a + \langle v, x \rangle) \Big| = 0 
  \\
  \max_{(a, v) \in \underline{d} f(x)} a = 0,
\end{split}
\end{equation}
where $\langle \cdot, \cdot \rangle$ is the inner product in $H$.
}

\item{if $X = \mathbb{R}^d$, then a hypodifferential of $f$ at $x \in \dom f$ is a compact subset
$\underline{d} f(x) \subset \mathbb{R}^{d + 1}$ satisfying equalities \eqref{eq:HypodiffDefHilbertCase} with 
$\langle \cdot, \cdot \rangle$ being the inner product in $\mathbb{R}^d$.
}
\end{enumerate}
\end{remark}

Unlike derivative/gradient in the smooth case and subdifferential in the nonsmooth case, hypodifferential of a convex
function is not uniquely defined even in the smooth case. In particular, if $f$ is G\^{a}teaux differentiable at a point
$x$ and $f'(x)$ is its G\^{a}teaux derivative at this point, then both $\{ (0, f'(x)) \}$ and 
$\co\{ (0, f'(x))), (a, x^*) \}$ for any $a < 0$ and $x^* \in X^*$ are hydifferentials of $f$ at $x$.

It turns out that closed convex functions are hypodifferentiable in the interior $\interior \dom f$ of their effective
domain. Moreover, one can give a complete characterization of all hypodifferentials of $f$ at any given point. Denote by
$\core Q$ the algebraic interior (radial kernel) of a set $Q \subset X$, that is, the set of all $x \in Q$ such that for
any $y \in X$ one can find $t_y > 0$ such that $x + t y \in Q$ for any $t \in [0, t_y]$.

\begin{theorem} \label{thrm:HypodiffCharacterization}
The following statements hold true:
\begin{enumerate}
\item{$f$ is hypodifferentiable at every point $x \in \interior \dom f$;}

\item{if $f$ is hypodifferentiable at a point $x \in \dom f$, then $x \in \core \dom f$, $f$ is subdifferentiable at
this point, the subdifferential $\partial f(x)$ is weak${}^*$ compact, and the directional derivative $f'(x, \cdot)$ of
$f$ at $x$ is a proper closed convex function;}

\item{if $f$ is hypodifferentiable at a point $x \in \dom f$, then a weak${}^*$ compact set 
$H \subset \mathbb{R} \times X^*$ is a hypodifferential of $f$ at $x$ if and only if it satisfies the following three
conditions:
\begin{enumerate}
\item{$a \le 0$ for any $(a, x^*) \in H$;}

\item{there exists $x^* \in X^*$ such that $(0, x^*) \in H$;}

\item{$\{ x^* \in X^* \mid (0, x^*) \in H \} = \partial f(x)$.}
\end{enumerate}
}
\end{enumerate}
\end{theorem}

\begin{proof}
For the sake of convenience, we divide the proof of the theorem into three parts corresponding to the three statements
of the theorem.

\textbf{Part 1.} Fix any $x \in \interior \dom f$. Let us show that $f$ is hypodifferentiable at $x$. Indeed, by
\cite[Corollary~I.2.5]{EkelandTemam} the function $f$ is continuous in a neighbourhood of $x$, since $X$ is a Banach
space, $f$ is a closed convex function, and $x \in \interior \dom f$. Therefore $f$ is subdifferentiable at $x$ by
\cite[Proposition~I.5.2]{EkelandTemam} and $\partial f(x)$ is a weak${}^*$ compact set by
\cite[Theorem~2.4.9]{Zalinescu}. Moreover, by \cite[Theorem~2.4.9]{Zalinescu} the function $f$ is directionally
differentiable at $x$ and for its directional derivative at this point, $f'(x, \cdot)$, the following equalities hold
true:
\[
  f'(x, u) = \lim_{\alpha \to +0} \frac{f(x + \alpha u) - f(x)}{\alpha} 
  = \max_{x^* \in \partial f(x)} \langle x^*, u \rangle \quad \forall u \in X.
\]
Define $\underline{d} f(x) = \{ 0 \} \times \partial f(x)$. Then the set $\underline{d} f(x)$ is weak${}^*$ compact,
the first equality in \eqref{eq:HypodiffDef} is satisfied for any $\Delta x \in X$, and 
$\max_{(a, x^*) \in \underline{d} f(x)} a = 0$. Thus, the function $f$ is hypodifferentiable at $x$ and the set
$\underline{d} f(x)$ is its hypodifferential at this point.

\textbf{Part 2.} Suppose now that $f$ is hypodifferentiable at a point $x \in \dom f$ and $\underline{d} f(x)$ is its
hypodifferential at this point. Introduce the function
\[
  \Phi(u) = \max_{(a, x^*) \in \underline{d} f(x)} (a + \langle x^*, u \rangle) \quad \forall u \in X.
\]
Note that from the definition of hypodifferential it follows that $\Phi(0) = 0$ and $\dom \Phi = X$. Moreover, $\Phi$ is
a closed convex function as the supremum of a family of affine function. Therefore, $\Phi$ is continuous on $X$,
subdifferentiable at the origin, $\partial \Phi(0)$ is a weak${}^*$ compact set, $\Phi$ is directionally differentiable
at the origin, and
\[
  \Phi'(0; u) = \lim_{\alpha \to +0} \frac{\Phi(\alpha u)}{\alpha} 
  = \max_{x^* \in \partial \Phi(0)} \langle x^*, u \rangle \quad \forall u \in X.
\]
Hence taking into account the first equality in \eqref{eq:HypodiffDef} one gets that for any $u \in X$
\begin{equation} \label{eq:DirectDerivViaMaxAproxSubdiff}
  f'(x, u) := \lim_{\alpha \to +0} \frac{f(x + \alpha u) - f(x)}{\alpha} = 
  \lim_{\alpha \to +0} \frac{\Phi(\alpha u)}{\alpha}
  = \max_{x^* \in \partial \Phi(0)} \langle x^*, u \rangle.
\end{equation}
Consequently, $\dom f'(x, \cdot) = X$, which implies that $x \in \core \dom f$, and $f'(x, \cdot)$ is a proper closed
convex function. Moreover, by \cite[Theorem~2.4.4, part~(i) and Corollary~2.4.15]{Zalinescu} one has
\begin{equation} \label{eq:SubdiffOfMaxApproximation}
  \partial f(x) = \partial f'(x, \cdot)(0) = \partial \Phi(0).
\end{equation}
Thus, $f$ is subdifferentiable at $x$ and $\partial f(x) = \partial \Phi(0)$ is a weak${}^*$ compact set.

\textbf{Part 3.} Suppose finally that $f$ is hypodifferentiable at a point $x \in \dom f$ and $\underline{d} f(x)$ is
its hypodifferential at this point. Then by definition one has $\max\{ a \mid (a, x^*) \in \underline{d} f(x) \} = 0$
or, equivalently, $a \le 0$ for any $(a, x^*) \in \underline{d} f(x)$ and there exists $x^* \in X^*$ such that
$(0, x^*) \in \underline{d} f(x)$. 

Note that by the theorem on the subdifferential of an infinite family of convex functions
\cite[Theorem~4.2.3]{IoffeTihomirov} one has
\[
  \partial \Phi(0) = \Big\{ x^* \in X^* \Bigm| (0, x^*) \in \underline{d} f(x) \Big\}.
\]
Hence with the use of \eqref{eq:SubdiffOfMaxApproximation} one gets that 
$\{ x^* \in X^* \mid (0, x^*) \in \underline{d} f(x) \} = \partial f(x)$.

Let now $H \subset \mathbb{R} \times X^*$ be a weak${}^*$ compat set satisfying the assumptions of the theorem. Let us
show that $H$ is a hypodifferential of $f$ at $x$. Indeed, introduce the function 
\[
  \Phi_H(u) = \max_{(a, x^*) \in H} (a + \langle x^*, u \rangle) \quad \forall u \in X. 
\]
Then $\Phi_H(0) = 0$, $\dom \Phi_H = X$ due to the fact that $H$ is weak${}^*$ compact, and $\Phi_H$ is a closed convex
function as the maximum of a family of affine functions. Therefore, $\Phi_H$ is continuous on $X$,
subdifferentiable and directionally differentiable at the origin, and
\[
  \Phi_H'(0; u) = \lim_{\alpha \to +0} \frac{\Phi_H(\alpha u)}{\alpha} 
  = \max_{x^* \in \partial \Phi_H(0)} \langle x^*, u \rangle \quad \forall u \in X.
\]
Moreover, applying \cite[Theorem~4.2.3]{IoffeTihomirov} and the assumptions on the set $H$ from the formulation of 
the theorem one gets
\[
  \partial \Phi_H(0) = \Big\{ x^* \in X^* \Bigm| (0, x^*) \in H \Big\} = \partial f(x). 
\]
Consequently, taking into account equalities \eqref{eq:DirectDerivViaMaxAproxSubdiff} and
\eqref{eq:SubdiffOfMaxApproximation} from the second part of the proof one can conclude that
\[
  \lim_{\alpha \to +0} \frac{f(x + \alpha u) - f(x)}{\alpha} 
  = \lim_{\alpha \to + 0} \frac{\Phi_H(\alpha u)}{\alpha} \quad \forall u \in X,
\]
which obviously yeilds
\[
  \lim_{\alpha \to +0} \frac{1}{\alpha} \Big| f(x + \alpha u) - f(x) 
  - \max_{(a, x^*) \in H} (a + \langle x^*, u \rangle) \Big| = 0
\]
for any $u \in X$. Therefore, $H$ is a hypodifferential of $f$ at $x$.
\end{proof}

\begin{corollary} \label{crlr:HypodiffCharacterization}
For the function $f$ to be hypodifferentiable at a point $x \in \dom f$ it is necessary and sufficient that 
$x \in \core \dom f$.
\end{corollary}

\begin{proof}
The necessity of the condition $x \in \core \dom f$ for the hypodifferentiability of $f$ at $x \in \dom f$ follows
directly from the second statement of Theorem~\ref{thrm:HypodiffCharacterization}. The sufficiency of this condition
follows from the first statement of Theorem~\ref{thrm:HypodiffCharacterization} and the fact that 
$\core \dom f = \interior \dom f$ by \cite[Theorem~2.2.20]{Zalinescu}, since $f$ is l.s.c. and $X$ is a Banach space.
\end{proof}

\begin{corollary}
Let $f$ be hypodifferentiable at a point $x_* \in X$ and $\underline{d} f(x_*)$ be its hypodifferential at this point.
Then $f$ attains a global minimum at $x_*$ if and only if $0 \in \underline{d} f(x_*)$.
\end{corollary}

Thus, in a sense, hypodifferential can be viewed as an extension of the subdifferential mapping into the space
$\mathbb{R} \times X^*$. Such extensions can be constructed in many different ways. As we will see below, there is
a natural way to construct such extensions that allows one to recover many standard  properties of the gradient of a
convex function in the nonsmooth case that subdifferentials do not have.

Since a hypodifferential is not uniquely defined, we restrict our consideration to a special class of hypodifferential
maps of nonsmooth convex functions that are especially convenient to work with. 

\begin{definition}
The function $f$ is called \textit{continuously} hypodifferentiable on a set $Q \subseteq \interior \dom f$, if $f$ is
hypodifferentiable at every point of this set and there exists a hypodifferential mapping 
$\underline{d} f \colon Q \rightrightarrows \mathbb{R} \times X^*$ that is continuous with respect to the
Pompeiu-Hausdorff distance. The function $f$ is said to be continuously hypodifferentiable near a point 
$x \in \interior \dom f$, if $f$ is continuously hypodifferentiable in a neighbourhood of $x$.
\end{definition}

\begin{definition}
Let $\underline{d} f \colon Q \rightrightarrows \mathbb{R} \times X^*$ be a hypodifferential mapping of $f$ on a set 
$Q \subseteq \dom f$. The hypodifferential mapping $\underline{d} f(\cdot)$ is called \textit{consistent} on the set
$Q$, if for any $x \in Q$ one has
\begin{equation} \label{eq:ConsistentHypodiff}
  f(y) - f(x) \ge a + \langle x^*, y - x \rangle \quad \forall y \in X, \: (a, x^*) \in \underline{d} f(x).
\end{equation}
\end{definition}

\begin{remark}
{(i)~Note that a hypodifferential map $\underline{d} f(\cdot)$ is consistent on a set $Q \subseteq \dom f$ if and 
only if
\[
  f(y) - f(x) \ge \max_{(a, x^*) \in \underline{d} f(x)} (a + \langle x^*, y - x \rangle) 
  \quad \forall y \in X, \: x \in Q.
\]
Roughly speaking, a consistent hypodifferential of $f$ is a weak${}^*$ compact family of affine minorants of $f$ that
defines a local approximation of $f$ at a given point in accordance with Definition~\ref{def:Hypodifferential}.
}

\noindent{(ii)~In the author's earlier paper on the method of hypodifferential descent
\cite{Dolgopolik_HypodiffDescent}, consistent hypodifferentials were called amenable. However, the term
\textit{consistent} seems to be more suitable.
}
\end{remark}

Let us give a simple example illustrating the definition of consistent and continuous hypodifferential mappings.

\begin{example} \label{ex:AbsoluteValue}
Let $X = \mathbb{R}$ and $f(x) = |x|$. Then for any $x, \Delta x \in \mathbb{R}$ one has
\[
  f(x + \Delta x) - f(x) = \max\{ x + \Delta x, - x - \Delta x \} - |x| 
  = \max\{ x - |x| + \Delta x, - x - |x| - \Delta x \},
\]
which implies that the set-valued map $\underline{d} f(x) := \co\{ (x - |x|, 1), (- x - |x|, - 1) \}$ is a
hypodifferential mapping of $f$ on $\mathbb{R}$.

Observe that for any $y, x \in \mathbb{R}$ one has
\begin{equation} \label{eq:AbsValueConsistency}
  |y| - |x| \ge y - |x| = (x - |x|) + (y - x) = a + v(y - x)
\end{equation}
for $(a, v) = (x - |x|, 1) \in \underline{d} f(x)$. Similarly, for any $y, x \in \mathbb{R}$ one has
\[
  |y| - |x| \ge - y - |x| = (- x - |x|) + (-1)(y - x) = a + v(y - x)
\]
for $(a, v) = (- x - |x|, -1) \in \underline{d} f(x)$. Summing up inequality \eqref{eq:AbsValueConsistency}
multiplied by $\alpha \in (0, 1)$ and the inequality above multiplied by $(1 - \alpha)$ one obtains that inequality
\eqref{eq:ConsistentHypodiff} holds true for all $x \in \mathbb{R}$. Thus, $\underline{d} f(\cdot)$ is a consistent
hypodifferential mapping on $\mathbb{R}$. Let us check that it is continuous.

Indeed, fix any $x, y \in \mathbb{R}$. By the definition of convex hull for any $(a, v) \in \underline{d} f(x)$ one has 
$(a, v) = ((2 \alpha - 1) x - |x|, 2 \alpha - 1)$ for some $\alpha \in [0, 1]$. Similarly, for any 
$(b, w) \in \underline{d} f(y)$ one has $(b, w) = ((2 \beta - 1) y - |y|, 2 \beta - 1)$. Therefore,
\begin{multline*}
  \max_{(a, v) \in \underline{d} f(x)} \dist((a, v), \underline{d} f(y)) 
  \\
  = \max_{\alpha \in [0, 1]} \min_{\beta \in [0, 1]} 
  \Big( \big| (2 \alpha - 1) x - |x| - ((2 \beta - 1)y - |y|) \big| + \big| (2 \alpha - 1) - (2 \beta - 1) \big| \Big).
\end{multline*}
Setting $\beta = \alpha$ one gets that
\begin{align*}
  \max_{(a, v) \in \underline{d} f(x)} \dist((a, v), \underline{d} f(y)) 
  &\le \max_{\alpha \in [0, 1]} \Big| (2 \alpha - 1) x - |x| - ((2 \alpha - 1)y - |y|) \Big| 
  \\
  &\le \max_{\alpha \in [0, 1]} \big( |2 \alpha - 1| + 1 \big) |x - y| = 2 |x - y|.
\end{align*}
Swapping $x$ and $y$ one can conclude that $d_{PH}(\underline{d} f(x), \underline{d} f(y)) \le 2 |x - y|$, that is, 
the hypodifferential mapping $\underline{d} f$ is continuous (furthermore, globally \textit{Lipschitz} continuous).
\end{example}

Comparing inequality \eqref{eq:ConsistentHypodiff} with the inequality from the definition of
$\varepsilon$-subdiffe\-ren\-ti\-al one obtains the following result.

\begin{proposition}
Let a hypodifferential map $\underline{d} f(\cdot)$ of the function $f$ be consistent on a set $Q \subseteq H$. Then
for any $x \in Q$ and $(a, x^*) \in \underline{d} f(x)$ one has $x^* \in \partial_{\varepsilon} f(x)$ 
for any $\varepsilon \ge - a$.
\end{proposition}

\begin{remark} \label{rmrk:GlobalHypodifferential}
From Theorem~\ref{thrm:HypodiffCharacterization} it follows that in the general case the inclusion
\[
  \Big\{ x^* \in X^* \Bigm| (a, x^*) \in \underline{d} f(x), \: -a \le \varepsilon \Big\} 
  \subseteq \partial_{\varepsilon} f(x)
\]
can be strict. This inclusion turns into an equality in the case when $\underline{d} f(\cdot)$ is a so-called
\textit{global} hypodifferential map of the function $f$ (see \cite{Dolgopolik_GlobOptCond}), that is, 
\[
  f(y) - f(x) = \max_{(a, v) \in \underline{d} f(x)} (a + \langle x^*, y - x \rangle) 
  \quad \forall y, x \in X.
\]
Below we will provide some nontrivial examples of nonsmooth convex functions whose naturally defined hypodifferential
mappings are global. Here we only note that the hypodifferential of the function $f(x) = |x|$ from
Example~\ref{ex:AbsoluteValue} is global.
\end{remark}

Let us point out a useful property of convex functions with bounded hypodifferential mapping that can be proved directly
with the use of standard results on convex functions \cite[Theorem~2.4.13]{Zalinescu} (see also
\cite[Theorem~24.7]{Rockafellar}) and Theorem~\ref{thrm:HypodiffCharacterization}, but is also a particular case of
\cite[Corollary~2]{Dolgopolik_CodiffDescent}.

\begin{proposition} \label{prp:BoundedHypodiff_LipContin}
Let $\underline{d} f$ be a hypodifferential map of the function $f$ on a convex set $Q \subseteq \dom f$. Suppose that
$L = \sup\{ \| x^* \| \mid (0, x^*) \in \underline{d} f(x), x \in Q \} < + \infty$. Then $f$ is Lipschitz continuous 
on $Q$ with constant $L$.
\end{proposition}

\subsection{Hypodifferential calculus}
\label{subsect:CalculusRules}

In this section we present main calculus rules for consistent continuous hypodifferential maps of nonsmooth convex
functions. Recall that $f \colon X \to \mathbb{R} \cup \{ + \infty \}$ is a closed convex function.

\begin{proposition}
Let $f$ be continuously G\^{a}teaux differentiable on a set $Q \subseteq \dom f$. Then the map 
$\underline{d} f(\cdot) = \{ (0, f'(\cdot)) \}$ is a consistent continuous hypodifferential mapping of $f$ on $Q$.
\end{proposition}

\begin{theorem}[Maximum over a compact set] \label{thrm:MaxInfFamily}
Let $T$ be a compact topological space and a function $f \colon X \times T \to \mathbb{R} \cup \{ + \infty \}$ be given.
Suppose that the following assumptions hold true:
\begin{enumerate}
\item{for any $x \in X$ the map $t \mapsto f(x, t)$ is continuous;}

\item{for any $t \in T$ the map $f_t(\cdot) = f(\cdot ,t)$ is a closed convex function;}

\item{the set $\interior \bigcap_{t \in T} \dom f_t$ is nonempty and $Q \subseteq \interior \bigcap_{t \in T} \dom f_t$
is a given nonempty set;} 

\item{for any $t \in T$ there exists a consistent hypodifferential mapping $\underline{d} f_t$ of the function $f_t$ 
on $Q$;}

\item{the set-valued map $(x, t) \mapsto \underline{d} f_t(x)$ is continuous on $Q \times T$.}
\end{enumerate}
Then the function $f(\cdot) = \sup_{t \in T} f(\cdot, t)$ is hypodifferentiable on $Q$ and the set-valued map
\begin{equation} \label{eq:InfMaxHypodiff}
  \underline{d} f(\cdot) 
  = \cl^* \co\Big\{ (f(\cdot, t) - f(\cdot), 0) + \underline{d} f_t(\cdot) \Bigm| t \in T \Big\}
\end{equation}
is a consistent hypodifferential of $f$ on $Q$. If, in addition, either $X$ is reflexive or the set $\cl D(x)$, where
\begin{equation} \label{eq:InfMaxPreHypodiff}
  D(x) := \co\{ (f(x, t) - f(x), 0) + \underline{d} f_t(x) \mid t \in T \},
\end{equation}
is weak${}^*$ closed for all $x \in Q$, then the function $f$ is continuously hypodifferentiable on $Q$ and 
the hypodifferential map \eqref{eq:InfMaxHypodiff} is continuous on $Q$.
\end{theorem}

\begin{proof}
For the sake of convenience, we divide the proof of the theorem into several parts.

\textbf{1. Hypodifferentiability.} By our assumption $Q \subseteq \interior \bigcap_{t \in T} \dom f_t$, which implies
that $Q \subseteq \interior \dom f$ due to the fact that the map $t \mapsto f(x, t)$ is continuous for any $x \in X$.
Hence by Theorem~\ref{thrm:HypodiffCharacterization} the function $f$ is hypodifferentiable on the set $Q$. 

\textbf{2. Equality for hypodifferential.} Fix any $x \in Q$. Let us prove that the set $\underline{d} f(x)$, defined
according to equality \eqref{eq:InfMaxHypodiff}, is a hypodifferential of $f$ at $x$. To this end, we will utilize
the characterization of hypodifferentials given in Theorem~\ref{thrm:HypodiffCharacterization}.

Let us first show that the set $D(x)$ (see \eqref{eq:InfMaxPreHypodiff}) is norm bounded. Then one can conclude that
the set $\underline{d} f(x) = \cl^* D(x)$ is weak${}^*$ compact.

For any $t \in T$ the hypodifferential $\underline{d} f_t(x)$ is bounded due to the fact that it is weak${}^*$ compact.
By the last assumption of the theorem the map $t \mapsto \underline{d} f_t(x)$ is continuous in the Pompeiu-Hausdorff
distance. Therefore, for any $t \in T$ there exists a neighbourhood of $\mathcal{U}_t$ of $t$ such that 
$d_{PH}(\underline{d} f_t(x), \underline{d} f_s(x)) \le 1$ for any $s \in \mathcal{U}_t$. Hence the set 
$\bigcup_{s \in \mathcal{U}_t} \underline{d} f_s(x)$ is bounded, which due to the compactness of the space $T$ implies
that the set $\bigcup_{t \in T} \underline{d} f_t(x)$ is bounded as well. In addition, the set 
$\{ (f(x, t) - f(x), 0) \mid t \in T \}$ is bounded due to the compactness of $T$ and continuity of the map 
$t \mapsto f(x, t)$. Consequently, the sets
\[
  \Big\{ (f(x, t) - f(x), 0) + \underline{d} f_t(x) \Bigm| t \in T \Big\}
\]
and $D(x)$ (see \eqref{eq:InfMaxPreHypodiff}) are bounded. Hence $\underline{d} f(x)$ is a weak${}^*$ compact set.

Since $f(x) = \sup_{t \in T} f(x, t)$, for any $(a, x^*) \in D(x)$ one obviously has $a \le 0$, and, furthermore,
\[
  \sup_{(a, x^*) \in D(x)} a \le \sup_{t \in T} (f(x, t) - f(x)) = \sup_{t \in T} f(x, t) - f(x) = 0.
\]
Consequently, $\max\{ a \mid (a, x^*) \in \underline{d} f(x) \} = 0$.

Thus, if we show that
\begin{equation} \label{eq:InfMaxSubdiffHypodiff}
  \partial f(x) = \Big\{ x^* \in X^* \Bigm| (0, x^*) \in \underline{d} f(x) \Big\},
\end{equation}
then by Theorem~\ref{thrm:HypodiffCharacterization} one can conclude that $\underline{d} f(x)$ is a hypodifferential 
of $f$ at $x$.

Indeed, by the theorem on the subdifferential of the supremum of an infinite family of convex functions
\cite[Theorem~4.2.3]{IoffeTihomirov} and Theorem~\ref{thrm:HypodiffCharacterization} one has
\begin{equation} \label{eq:InfMaxSubdiff}
  \partial f(x) = \cl^* \bigcup_{t \in T(x)} \partial f_t(x)
  = \cl^* \Big\{ x^* \in X^* \Bigm| (0, x^*) \in \underline{d} f_t(x), \: t \in T(x) \Big\},
\end{equation}
where $T(x) = \{ t \in T \mid f(x, t) = f(x) \}$. To find a more convenient expression for the set on the right-hand
side of equality \eqref{eq:InfMaxSubdiffHypodiff}, introduce the function
\[
  \Phi(y) = \max_{(a, x^*) \in \underline{d} f(x)} (a + \langle x^*, y \rangle) \quad \forall y \in X.
\]
Clearly, $\Phi$ is a closed real-valued convex function and $\underline{d} f(x)$ is a hypodifferential of this function
at zero (note that $\Phi(0) = 0$, since $\max\{ a \mid (a, x^*) \in \underline{d} f(x) \} = 0$). Therefore,
by Theorem~\ref{thrm:HypodiffCharacterization} one has
\begin{equation} \label{eq:MaxHypodiffEnvelopeSubdiff}
  \partial \Phi(0) = \Big\{ x^* \in X^* \Bigm| (0, x^*) \in \underline{d} f(x) \Big\}.
\end{equation}
On the other hand, observe that for any $y \in X$ one has
\[
  \Phi(y) = \sup_{t \in T} \Psi(y, t), \quad 
  \Psi(y, t) := f(x, t) - f(x) + \max_{(a, x^*) \in \underline{d} f_t(x)} (a + \langle x^*, y \rangle \rangle).
\]
Clearly, for any $t \in T$ the function $y \mapsto \Psi(y, t)$ is a closed real-valued convex function that is bounded
on bounded sets (due to the boundedness of $\underline{d} f_t(x)$) and therefore continuous on $X$ by
\cite[Corollary~I.2.5]{EkelandTemam}. In turn, for any $y \in X$ the function $t \mapsto \Psi(y, t)$ is, as one can
readily
check, continuous due to the continuity of the maps $t \mapsto f(x, t)$ and $t \mapsto \underline{d} f_t(x)$. Note
finally that $\Psi(0, t) = \Phi(0)$ if and only if $t \in T(x)$, since $\Psi(0, t) = f(x, t) - f(x)$. Therefore, by the
theorem on the subdifferential of the supremum of an infinite family of convex functions
\cite[Theorem~4.2.3]{IoffeTihomirov} one has
\begin{equation} \label{eq:MaxHypodiffEnvelopeSubdiff_Alt}
  \partial \Phi(0) = \cl^* \Big\{ \partial_y \Psi(0, t) \Bigm| t \in T(x) \Big\},
\end{equation}
where $\partial_y \Psi(0, t)$ is the subdifferential of the map $\Psi(\cdot, t)$ at the origin. 

Taking into account the fact that for any $t \in T(x)$ one has
\[
  \Psi(y, t) = \max_{(a, x^*) \in \underline{d} f_t(x)} (a + \langle x^*, y \rangle) \quad \forall y \in X,
\]
one can conclude that $\underline{d} f_t(x)$ is a hypodifferential of $\Psi(\cdot, t)$ at the origin. Hence
\[
  \partial_y \Psi(0, t) = \Big\{ x^* \in X^* \Bigm| (0, x^*) \in \underline{d} f_t(x) \Big\} 
  \quad \forall t \in T(x)
\]
by Theorem~\ref{thrm:HypodiffCharacterization}. Now, combining the equality above with equalities
\eqref{eq:MaxHypodiffEnvelopeSubdiff}, \eqref{eq:MaxHypodiffEnvelopeSubdiff_Alt}, and \eqref{eq:InfMaxSubdiff} one
obtains that equality \eqref{eq:InfMaxSubdiffHypodiff} holds true. Thus, $\underline{d} f(x)$ is indeed a
hypodifferential of $f$ at $x$.

\textbf{3. Consistency.} Let us check that the hypodifferential map $\underline{d} f(\cdot)$ is consistent. Indeed, for
any $x \in Q$, $t \in T$, $(a, x^*) \in \underline{d} f_t(x)$, and $y \in X$ one has
\begin{equation} \label{eq:InfMaxConsistency}
  f(y) - f(x) \ge f(y, t) - f(x) \ge f(x, t) - f(x) + a + \langle x^*, y - x \rangle,
\end{equation}
since $\underline{d} f_t(x)$ is a consistent hypodifferential of $f_t$ at $x$ and $f(y) = \sup_{t \in T} f(y, t)$. 

Fix now any $x \in Q$ and $(a, x^*) \in D(x)$. Then by the definition of convex hull there exists $n \in \mathbb{N}$,
$t_i \in T$, $(a_i, x_i^*) \in \underline{d} f_{t_i}(x)$, and $\alpha_i \ge 0$, $i \in \{ 1, \ldots, n \}$, such that
\[
  (a, x^*) = \sum_{i = 1}^n \alpha_i (f(x, t_i) - f(x) + a_i, x_i^*), \quad 
  \sum_{i = 1}^n \alpha_i = 1.
\]
Hence applying inequality \eqref{eq:InfMaxConsistency} one gets that
\begin{equation} \label{eq:InfMaxConsistency_2}
\begin{split}
  f(y) - f(x) &= \sum_{i = 1}^n \alpha_i (f(y) - f(x)) 
  \\
  &\ge \sum_{i = 1}^n \alpha_i( f(x, t_i) - f(x) + a_i + \langle x^*, y - x \rangle )
  = a + \langle x^*, y - x \rangle
\end{split}
\end{equation}
for any $y \in X$.

Finally, fix any $(a, x^*) \in \underline{d} f(x)$, $y \in X$, and $\varepsilon > 0$. Since 
$\underline{d} f(x) = \cl^* D(x)$, one can find $(b, y^*) \in D(x)$ such that 
$|a - b| + |\langle y^* - x^*, y - x \rangle| < \varepsilon$. Consequently, with the use of inequality
\eqref{eq:InfMaxConsistency_2} one gets that
\[
  f(y) - f(x) \ge b + \langle y^*, y - x \rangle \ge a + \langle x^*, y - x \rangle - \varepsilon,
\]
which implies that 
\[
  f(y) - f(x) \ge a + \langle x^*, y - x \rangle \quad \forall y \in X
\] 
due to the fact that $\varepsilon > 0$ was chosen arbitrarily. Thus, the inequality above is satisfied for all 
$(a, x^*) \in \underline{d} f(x)$ and $y \in X$, which means that the hypodifferential $\underline{d} f(x)$ is
consistent.

\textbf{4. Continuity.} Let us finally prove the continuity of the hypodifferential map $\underline{d} f(\cdot)$ 
on $Q$. Indeed, as was noted above, $Q \subseteq \interior \dom f$. Therefore, the function $f$ is continuous on $Q$ 
by \cite[Corollary~I.2.5]{EkelandTemam}. Consequently, under the assumptions of the theorem the set-valued map  
\[
  \Big\{ (f(\cdot, t) - f(\cdot), 0) + \underline{d} f_t(x) \Bigm| t \in T \Big\}
\]
is continuous in the Pompeiu-Hausdorff metric on $Q$ by Proposition~\ref{prp:UnionContinuous}, which by 
Proposition~\ref{prp:ConvexHullContinuous} implies that the set-valued map $D(\cdot)$ is continuous on $Q$. Finally, 
the continuity of $\underline{d} f(\cdot) = \cl^* D(\cdot)$ then trivially follows from the continuity of $D(\cdot)$ in
the case when $\cl D(\cdot)$ is weak${}^*$ closed, since in this case $\underline{d} f(\cdot) = \cl D(\cdot)$ and
$d_{PH}(\cl D(x), \cl D(y)) = d_{PH}(D(x), D(y))$ for any $x, y \in Q$. In the case when $X$ is reflexive the continuity
of $\underline{d} f(\cdot)$ follows directly from Proposition~\ref{prp:WeakStarClosureContinuous}.
\end{proof}

\begin{corollary}[Finite maximum]
Let $f_i \colon X \to \mathbb{R} \cup \{ + \infty \}$, $i \in I = \{ 1, \ldots, n \}$, be closed convex functions such
that $\bigcap_{i \in I} \interior \dom f_i \ne \emptyset$. Suppose that there exist consistent continuous
hypodifferential mappings $\underline{d} f_i(\cdot)$ of these functions on a nonempty set 
$Q \subseteq \bigcap_{i \in I} \interior \dom f_i$. Then the function $f = \max_{i \in I} f_i$ is continuously 
hypodifferentiable on $Q$ and the set-valued map
\begin{equation} \label{eq:FiniteMaxHypodiff}
  \underline{d} f(\cdot) 
  = \co\Big\{ \{ (f_i(\cdot) - f(\cdot), 0) \} + \underline{d} f_i(\cdot) \Bigm| i \in I \Big\}
\end{equation}
is a consistent continuous hypodifferential of $f$ on $Q$.
\end{corollary}

\begin{proof}
Put $T = \{ 1, \ldots, n \}$ and endow it with discrete topology. For any $x \in X$ and $i \in T$ define 
$f(x, i) = f_i(x)$. Then, as one can readily check, all assumptions of the previous theorem hold true. It remains to
note that the set $\underline{d} f(x)$, defined according to equality \eqref{eq:FiniteMaxHypodiff}, is weak${}^*$ closed
for any $x \in Q$ as the convex hull of a \textit{finite} number of weak${}^*$ compact convex sets.
\end{proof}

Recall that a function $f \colon \mathbb{R}^n \to \mathbb{R}$ is called non-decreasing, if $f(x) \le f(y)$ for any 
$x, y \in \mathbb{R}^n$ such that $x \le y$, where the inequality $x \le y$ is understood coordinate-wise.

\begin{theorem}[Outer composition with differentiable function] \label{thm:OuterComposition}
Let $f_i \colon X \to \mathbb{R} \cup \{ + \infty \}$, $i \in I = \{ 1, \ldots, n \}$, be closed convex functions and
there exist consistent continuous hypodifferential mappings $\underline{d} f_i(\cdot)$ of these functions on a set 
$Q \subseteq \bigcap_{i \in I} \interior \dom f_i$. Let also  $g \colon \mathbb{R}^n \to \mathbb{R}$, $g = g(y)$, be a
non-decreasing convex function that is continuously differentiable in a neighbourhood of the set $f(Q)$, where 
$f = (f_1, \ldots, f_n)$. Then the convex function $h(\cdot) = g(f_1(\cdot), \ldots, f_n(\cdot))$ is defined in 
a neighbourhood of $Q$, continuously hypodifferentiable on $Q$, and the set-valued map
\begin{equation} \label{eq:OuterCompositionHypodiff}
  \underline{d} h(\cdot) = \sum_{i = 1}^n \frac{\partial g(f(\cdot))}{\partial y_i} \underline{d} f_i(\cdot)
\end{equation}
is a consistent continuous hypodifferential of $h$ on $Q$ (here $h(x) = + \infty$, if $f_i(x) = + \infty$ for some 
$i \in I$).
\end{theorem}

\begin{proof}
The convexity of the composition $h = g \circ f$ follows from the fact that $g$ is non-decreasing and can be easily
verified directly. Note also that the partial derivatives $\partial g(f(\cdot))/\partial y_i$ are nonnegative, since $g$
is non-decreasing.

The fact that $h$ is hypodifferentiable on $Q$ and \eqref{eq:OuterCompositionHypodiff} is its hypodifferential mapping
on this set is proved in exactly the same way as \cite[Theorem~2]{Dolgopolik_CodiffDescent}. The continuity of this
mapping follows from Propositions~\ref{prp:SumContinuous} and \ref{prp:ProductContinuous} and the fact that $f_i$ are
continuous on $Q$ by Proposition~\ref{prp:BoundedHypodiff_LipContin}.

Let us finally check that the hypodifferential mapping \eqref{eq:OuterCompositionHypodiff} is consistent on $Q$. Indeed,
choose any $x \in Q$ and $(a, x^*) \in \underline{d} h(x)$. Then there exists $(a_i, x_i^*) \in \underline{d} f_i(x)$,
$i \in I$ such that
\[
  (a, x^*) = \sum_{i = 1}^n \frac{\partial g(f(x))}{\partial y_i} (a_i, x_i^*).
\]
Choose any $y \in X$. Bearing in mind the convexity of $g$ one has
\[
  h(y) - h(x) = g(f(y)) - g(f(x)) \ge \sum_{i = 1}^n \frac{\partial g(f(x))}{\partial y_i} 
  \big( f_i(y) - f_i(x) \big).
\]
Since the hypodifferentials $\underline{d} f_i$ are consistent, one has
\[
  f_i(y) - f_i(x) \ge a_i + \langle x_i^*, y - x \rangle \quad \forall i \in I.
\]
Hence taking into account the fact that the partial derivatives $\partial g(f(x))/\partial y_i$ are nonnegative one
obtains that
\[
  h(y) - h(x) \ge \sum_{i = 1}^n \frac{\partial g(f(x))}{\partial y_i}
  \big( a_i + \langle x_i^*, y - x \rangle \big) = a + \langle x^*, y - x \rangle,
\]
which completes the proof.
\end{proof}

\begin{remark}
The theorem above can be slightly generalised in the following way. Namely, suppose that $K \subset \mathbb{R}^n$ is a
convex cone such that if $x \in K$ and $x \le y$, then $y \in K$ (in particular, $K$ can be the nonnegative orthant).
Then under the assumption that $f$ maps $X$ to $K$ one can suppose that $g$ is defined and nondecreasing only on $K$.
\end{remark}

\begin{corollary}[Conic combination]
Let $f_i \colon X \to \mathbb{R} \cup \{ + \infty \}$, $i \in I = \{ 1, \ldots, n \}$, be closed convex functions and
there exist consistent continuous hypodifferential mappings $\underline{d} f_i(\cdot)$ of these functions on a set 
$Q \subseteq \bigcap_{i \in I} \dom f_i$. Then for any $\lambda_i \ge 0$, $i \in I$, the function 
$f = \sum_{i \in I} \lambda_i f_i$ is continuously hypodifferentiable on $Q$ and the set-valued map
\[
  \underline{d} f(\cdot) =\sum_{i \in I} \lambda_i \underline{d} f_i(\cdot)
\]
is a consistent continuous hypodifferential of $f$ on $Q$.
\end{corollary}

\begin{corollary}
Let there exist a consistent continuous hypodifferential mapping $\underline{d} f(\cdot)$ of $f$ on a set 
$Q \subseteq \interior \dom f$. Then the functions $\max\{ 0, f(\cdot) \}^p$, $p > 1$, and $\exp(f(\cdot))$ are
continuously hypodifferentiable on $Q$ and the set-valued maps
\[
  \max\{ 0, f(\cdot) \}^{p - 1} \underline{d} f(\cdot), \quad 
  \exp(f(\cdot)) \underline{d} f(\cdot)
\]
are their consistent continuous hypodifferential mappings on the set $Q$.
\end{corollary}

The following result can be readily verified directly

\begin{proposition}[Composition with affine map] \label{prp:AffineShiftHypodiff}
Let $Y$ be a real Banach space, $A \colon Y \to X$ be a continuous linear operator, $b \in X$ be fixed, and there
exist a consistent continuous hypodifferential map $\underline{d} f(\cdot)$ of $f$ on a set 
$Q \subseteq \interior \dom f$. Then the function $g(y) = f(A y + b)$, $y \in Y$, is continuously hypodifferentiable on
the set $A^{-1}(Q - b)$ and the set-valued map
\begin{equation} \label{eq:AffineShiftHypodiff}
  \underline{d} g(\cdot) 
  = \Big\{ (a, A^* x^*) \in \mathbb{R} \times Y^* \Bigm| (a, x^*) \in \underline{d} f(\cdot) \Big\},
\end{equation}
is a consistent continuous hypodifferential map of $g$ on $A^{-1}(Q - b)$, where $A^*$ is the adjoint operator of $A$.
\end{proposition}

\subsection{Examples}
\label{subsect:Examples}

Let us give several nontrivial examples of hypodifferentials of nonsmooth convex functions. It should be noted that the
hypodifferentials computed in all examples below are \textit{global} (see Remark~\ref{rmrk:GlobalHypodifferential}).

\begin{example}[Polyhedral function]
Let $X$ be the space $\mathbb{R}^d$ equipped with the Euclidean norm $|\cdot|$. Let also 
$f \colon \mathbb{R}^d \to \mathbb{R}$ be a polyhedral convex function, that is, the epigraph of $f$ is the intersection
of a finite family of closed half-spaces. Then there exist $n \in \mathbb{N}$, 
$a_i \in \mathbb{R}$ and $v_i \in \mathbb{R}^d$, $i \in I = \{ 1, \ldots, n \}$, such that
\[
  f(x) = \max_{i \in I} (a_i + \langle v_i, x \rangle) \quad \forall x \in \mathbb{R}^d
\]
(see \cite[Section~19]{Rockafellar}). Consequently, for any $x, \Delta x \in \mathbb{R}^d$ one has
\begin{equation} \label{eq:PolyFunc_GlobalHypodiff}
\begin{split}
  f(x + \Delta x) - f(x) 
  &= \max_{i \in I} \big( a_i + \langle v_i, x \rangle - f(x) + \langle v_i, \Delta x \rangle \big)
  \\
  &= \max_{(a, v) \in \underline{d} f(x)} (a + \langle v, x \rangle),
\end{split}
\end{equation}
where
\begin{equation} \label{eq:PolyFuncHypodiff}
  \underline{d} f(\cdot) = \co\Big\{ \big( a_i + \langle v_i, \cdot \rangle - f(\cdot), v_i \big) \in \mathbb{R}^{d + 1}
  \Bigm| i \in I \Big\}.
\end{equation}
Clearly, $\underline{d} f(x)$ is a convex compact set and $\max\{ a \mid (a, v) \in \underline{d} f(x) \} = 0$ for
any $x \in \mathbb{R}^d$. Therefore, $f$ is hypodifferentiable on $\mathbb{R}^d$ and $\underline{d} f(\cdot)$ is its
hypodifferential map on $\mathbb{R}^d$ that is consistent and global due to equality \eqref{eq:PolyFunc_GlobalHypodiff}.

Let us show that the set-valued map $\underline{d} f(\cdot)$ is continuous. Indeed, fix any $x, y \in \mathbb{R}^d$ and
choose any $(a, v) \in \underline{d} f(x)$. Then by definition there exist $\alpha_i \ge 0$, $i \in I$, such that
\[
  (a, v) = \sum_{i = 1}^n \alpha_i (a_i + \langle v_i, x \rangle - f(x), v_i), \quad \sum_{i = 1}^n \alpha_i = 1.
\]
Define
\[
  b = \sum_{i = 1}^n \alpha_i (a_i + \langle v_i, y \rangle - f(y)).
\]
Then $(b, v) \in \underline{d} f(y)$ and taking into account the fact that $f$ is globally Lipschitz continuous with
Lipschitz constant $L = \max_{i \in I} |v_i|$ (see Proposition~\ref{prp:BoundedHypodiff_LipContin}) one obtains that
\begin{align*}
  \dist( (a, v), \underline{d} f(y) ) &\le \big| (a, v) - (b, v) \big| = |a - b|
  \\
  &\le \sum_{i = 1}^n \alpha_i \big( |\langle v_i, x - y \rangle| + |f(x) - f(y)| \big)
  \\
  &\le \max_{i \in I} |v_i| |x - y| + L |x - y| = 2 L |x - y|.
\end{align*}
Taking the supremum over all $(a, v) \in \underline{d} f(x)$ and then repeating the same argument with $x$ and $y$
swapped one obtains that
\[
  d_{PH}\big( \underline{d} f(x), \underline{d} f(y) \big) \le 2 \Big( \max_{i \in I} |v_i| \Big) |x - y| 
  \quad \forall x, y \in \mathbb{R}^d.
\]
Thus, $f$ is continuously hypodifferentiable on $\mathbb{R}^d$ and the set-valued map \eqref{eq:PolyFuncHypodiff} is its
consistent continuous hypodifferential map on $\mathbb{R}^d$.
\end{example}

\begin{example}[Sublinear function] \label{ex:SublinearFunction}
Let $f \colon X \to \mathbb{R}$ be a closed sublinear function. Then the subdifferential $\partial f(0)$ is weak${}^*$
compact and
\[
  f(x) = \max_{x^* \in \partial f(0)} \langle x^*, x \rangle \quad \forall x \in X
\]
(see, e.g. \cite[Theorem~2.4.9]{Zalinescu}). Consequently, for any $x, \Delta x \in X$ one has
\begin{equation} \label{eq:SublinFunc_GlobalHypodiff}
  f(x + \Delta x) - f(x) = \max_{x^* \in \partial f(0)} 
  \big( \langle x^*, x \rangle - f(x) + \langle x^*, \Delta x \rangle \big)
  = \max_{(a, x^*) \in \underline{d} f(x)} (a + \langle x^*, x \rangle),
\end{equation}
where
\begin{equation} \label{eq:SublinearFuncHypodiff}
  \underline{d} f(\cdot) = \Big\{ (\langle x^*, \cdot \rangle - f(\cdot), x^*) \in \mathbb{R} \times X^* \Bigm|
  x^* \in \partial f(0) \Big\}
\end{equation}
Observe that for any $x \in X$ one has
\[
  \max_{(a, x^*) \in \underline{d} f(x)} a = \max_{x^* \in \partial f(0)} (\langle x^*, x \rangle - f(x))
  = f(x) - f(x) = 0.
\]
Moreover, $\underline{d} f(x) = \mathcal{T}(\partial f(0))$ with
\[
  \mathcal{T}(x^*) = \Big( \langle x^*, x \rangle - f(x), x^* \Big) \in \mathbb{R} \times X^*
  \quad \forall x^* \in X^*.
\]
Clearly, $\mathcal{T}$ is an affine map, which implies that the set $\underline{d} f(x)$ is convex. Moreover, as one
can readily check, $\mathcal{T}$ continuously maps $X^*$ endowed with the weak${}^*$ topology to 
$\mathbb{R} \times X^*$ equipped with the weak${}^*$ topology. Therefore the set $\underline{d} f(x)$ is weak${}^*$
compact as the continuous image of a compact set. Thus, $f$ is hypodifferentiable on $X$ and the set-valued map
\eqref{eq:SublinearFuncHypodiff} is its consistent hypodifferential mapping on $X$ that is, furthermore, global (see
equality \eqref{eq:SublinFunc_GlobalHypodiff}).

Let us prove that the map $\underline{d} f$ is continuous. Indeed, note that $f$ is globally Lipschitz continuous with
Lipschitz constant $L = \sup\{ \| x^* \| \mid x^* \in \partial f(0) \}$ (see
Proposition~\ref{prp:BoundedHypodiff_LipContin}). Fix any $x, y \in X$ and choose $(a, x^*) \in \underline{d} f(x)$.
Then by definion $x^* \in \partial f(0)$ and $a = \langle x^*, x \rangle - f(x)$. Define 
$b = \langle x^*, y \rangle - f(y)$. Then $(b, x^*) \in \underline{d} f(y)$ and
\begin{align*}
  \dist\big( (a, x^*), \underline{d} f(y) \big) &\le \| (a, x^*) - (b, x^*) \| = |a - b| 
  \\
  &\le |\langle x^*, x - y \rangle| + |f(x) - f(y)| \le 2 L \| x - y \|
\end{align*}
Hence taking the supremum over all $(a, x^*) \in \underline{d} f(x)$ and then repeating the same argument with $x$ and
$y$ swapped one obtains that
\[
  d_{PH}(\underline{d} f(x), \underline{d} f(y)) \le 2 \Big( \sup_{x^* \in \partial f(0)} \| x^* \| \Big) \| x - y \|.
\]
Thus, $f$ is continuously hypodifferentiable and the set-valued map \eqref{eq:SublinearFuncHypodiff} is its consistent
continuous hypodifferential map on $X$.
\end{example}

\begin{example}[Norm of an affine map]
Suppose that $Y$ is a real Banach space, $A \colon X \to Y$ is a bounded linear operator, and $f(x) = \| Ax + b \|$ for
some $b \in Y$. Then by the previous example and Proposition~\ref{prp:AffineShiftHypodiff} the function $f$ is
continuously hypodifferentiable on $X$ and the set-valued map
\[
  \underline{d} f(x) 
  = \Big\{ \Big( \langle x^*, A x + b \rangle - \| A x + b \|, A^* x^* \Big) \in \mathbb{R} \times X^*
  \Bigm| \| x^* \| \le 1 \Big\} \quad \forall x \in X
\]
is its consistent continuous hypodifferential mapping on $X$. Moreover, this hypodifferential map is global, and
arguing as in the previous example one can readily check that
\[
  d_{PH}\big( \underline{d} f(x), \underline{d} f(y) \big) \le 2 \| A \| \| x - y \| \quad \forall x, y \in X.
\]
\end{example}

\begin{example}[Maximal eigenvalue]
Let $X$ be the space $\mathbb{S}^{\ell}$ of all real symmetric matrices of order $\ell$ equipped with the inner product
$\langle A, B \rangle = \trace(AB)$ and the corresponding norm (here $\trace(A)$ is the trace of a square matrix $A$).
Let $f(A) = \lambda_{\max}(A)$, $A \in \mathbb{S}^{\ell}$, be the maximal eigenvalue function. Clearly, $f$ is a
closed
sublinear function. Taking into the well-known representation
\[
  \lambda_{\max}(A) = \max_{v \in \mathbb{B}_{\ell}} \langle v, A v \rangle \quad \forall A \in \mathbb{S}^{\ell}
\]
one can readily check that $\partial f(0) = \{ v^T v \mid v \in \mathbb{B}_{\ell} \}$, where $\mathbb{B}_{\ell}$ is the
unit ball in $\mathbb{R}^{\ell}$. Therefore, by Example~\ref{ex:SublinearFunction} the function $f$ is continuously
hypodifferentiable on $\mathbb{S}^{\ell}$ and the set-valued map
\[
  \underline{d} f(A) = \Big\{ (\langle v, A v \rangle - \lambda_{\max}(A), v^T v) 
  \in \mathbb{R} \times \mathbb{S}^{\ell} \Bigm| v \in \mathbb{B}_{\ell} \Big\}
  \quad \forall A \in \mathbb{S}^{\ell}
\]
is its consistent continuous hypodifferential mapping on $X$ that is, furthermore, global. In addition, one can easily
check that
\[
  d_{PH}\big( \underline{d} f(A), \underline{d} f(B) \big) \le 2 \| A - B \| \quad 
  \forall A, B \in \mathbb{S}^{\ell}.
\]
\end{example}

\begin{example}[Distance to a cone]
Let $K \subset X$ be a closed convex cone and $f(x) = \dist(x, K)$ for all $x \in X$. Bearing in mind the fact that $K$
is a convex cone one can easily verify that $f$ is a closed sublinear function. By \cite[Example~2.130]{BonnansShapiro}
one has
\[
  \partial f(0) = \big\{ x^* \in K^* \bigm| \| x^* \| \le 1 \big\}, \quad
  K^* = \Big\{ x^* \in X^* \Bigm| \langle x^*, x \rangle \le 0 \: \forall x \in K \Big\}
\]
(here $K^*$ is the polar cone of $K$). Hence by Example~\ref{ex:SublinearFunction} the function $f$ is continuously
hypodifferentiable on $X$ and the set-valued map
\[
  \underline{d} f(x) = \Big\{ (\langle x^*, x \rangle - \dist(x, K), x^*) \in \mathbb{R} \times X^* \Bigm|
  x^* \in K^*, \: \| x^* \| \le 1 \Big\}
\]
is its global and consistent continuous hypodifferential mapping on $X$. Moreover, in this case
\[
  d_{PH}\big( \underline{d} f(x), \underline{d} f(y) \big) \le 2 \| x - y \|
  \quad \forall x, y \in X.
\]
\end{example}

\begin{remark}
Let us note that a consistent continuous hypodifferential map of the convex functional
\[
  F(u) = \int_{\Omega} f(u(x), \nabla u(x), x) \, dx \quad \forall u \in W^{1, p}(\Omega),
\]
defined on the Sobolev space $W^{1, p}(\Omega)$, can be computed with the use of
\cite[Theorem~5.1]{Dolgopolik_COCV} and \cite[Theorem~3.3]{Dolgopolik_COCV_2} under the assumptions that 
$f = f(u, \xi, x)$ is convex in $(u, \xi)$ for a.e. $x \in \Omega$ and the function $f$ along with its hypodifferential
mapping satisfy some natural growth conditions.
\end{remark}

\subsection{Lipschitzian approximations and Lipschitz continuous hypodifferentials}
\label{subsect:LipschitzProperty}

The Lipschitz continuity of the gradient of the objective function plays a fundamental role in the design and analysis
of numerical methods of convex optimization \cite{Nesterov_book,Bubeck,DvurechenskyShternStaudigl,Teboulle}. In the case
when $X$ is a Hilbert space and $f$ is differentiable, the Lipschitz continuity of the gradient of $f$ on a convex set
$Q$ implies that
\[
  \Big| f(y) - f(x) - \langle \nabla f(x), y - x \rangle \Big| \le \frac{L}{2} \| y - x \|^2
  \quad \forall x, y \in Q
\]
(see, e.g. \cite[Lemma~1.2.3]{Nesterov_book}). This inequality is sometimes sufficient by itself for convergence
analysis of optimization methods and is often exploited, for example, in the context of weakly convex optimization
(see, e.g. \cite{AtenasSagastizabal,DanilovaDvurechnsky} and the references therein). In
\cite{Dolgopolik_HypodiffDescent}, a natural extension of this inequality to the case of nonsmooth hypodifferentiable
convex functions was proposed and studied. 

\begin{definition} \label{def:LipApprox}
Let $Q \subset \interior \dom f$ be a nonempty set and $\underline{d} f$ be a hypodifferential mapping of $f$ on $Q$.
This mapping is called a \textit{Lipschitzian approximation} of $f$ on $Q$ with Lipschitz constant $L > 0$, if
\begin{equation} \label{eq:LipApprox}
  \Big| f(y) - f(x) - \max_{(a, x^*) \in \underline{d} f(x)} \big( a + \langle x^*, y - x \rangle \big) \Big|
  \le \frac{L}{2} \| x - y \|^2
\end{equation}
for all $x, y \in Q$.
\end{definition}

Observe that if a hypodifferential mapping $\underline{d} f(\cdot)$ is consistent on $Q$, then it is a Lipschitzian
approximation of $f$ on $Q$ with Lipschitz constant $L > 0$ if and only if
\[
  f(y) \le 
  \varphi_x(y) := f(x) + \max_{(a, x^*) \in \underline{d} f(x)} \big( a + \langle x^*, y - x \rangle \big)
  + \frac{L}{2} \| y - x \|^2
\]
for all $x, y \in Q$ (that is, the function $\varphi_x(\cdot)$ is a majorant of $f$ on $Q$ for any $x \in Q$).

One can also consider the property of the Lipschitz continuity of hypodifferential, which in the convex setting can be
viewed as a natural extension of the Lipschitz continuity of the gradient to the nonsmooth case.

\begin{definition}
Let $Q \subset \interior \dom f$ be a nonempty set and $\underline{d} f$ be a hypodifferential mapping of $f$ on $Q$.
This mapping is said to be \textit{Lipschitz continuous} on $Q$ with Lipschitz constant $L > 0$, if
\[
  d_{PH}\big( \underline{d} f(x), \underline{d} f(y) \big) \le L \| x - y \| \quad \forall x, y \in Q
\]
or, equivalently,
\[
  \underline{d} f(y) \subseteq \underline{d} f(x) + L \| y - x \| \mathbb{B}_{\mathbb{R} \times X^*}
  \quad \forall x, y \in Q,
\]
where $\mathbb{B}_{\mathbb{R} \times X^*}$ is the unit ball in $\mathbb{R} \times X^*$. If a hypodifferential map
$\underline{d} f$ is both Lipschitz continuous on $Q$ and a Lipschitzian approximation of $f$ on $Q$, then it is said
to have \textit{the Lipschitz property} on $Q$.
\end{definition}

The Lipschitz continuity of a hypodifferential map and the property of it being a Lipschitzian approximation seem to be
closely related, as we will show below. However, it is unclear whether these properties are equivalent in the general
nonsmooth case. We leave the difficult question of whether these properties are equivalent (at least under some
additional assumptions) as an interesting problem for future research. Here we only prove a partial result showing that
a certain upper estimate of the function
\[
  f(y) - f(x) - \max_{(a, x^*) \in \underline{d} f(x)} (a + \langle x^*, y - x \rangle)
\]
holds true for any Lipschitz continuous hypodifferential map $\underline{d} f(\cdot)$.

\begin{proposition} \label{prp:LipContinHypodiff_Majorant}
Let $\underline{d} f$ be a consistent hypodifferential map of $f$ that is defined and Lipschitz continuous with
Lipschitz constant $L > 0$ on a convex set $Q \subseteq \interior \dom f$. Then for any $x, y \in Q$ one has
\[
  0 \le f(y) - f(x) - \max_{(a, x^*) \in \underline{d} f(x)} (a + \langle x^*, y - x \rangle)
  \le L \| y - x \| \big( 1 + \| y - x \| \big).
\]
\end{proposition}

\begin{proof}
Fix any $x, y \in Q$ and $\varepsilon > 0$. By \cite[Proposition~2]{Dolgopolik_CodiffDescent} there exists 
$z \in \co\{ x, y \}$ and $(0, x_z^*) \in \underline{d} f(z)$ such that $f(y) - f(x) = \langle x_z^*, y - x \rangle$. 
By our assumption $d_{PH}(\underline{d} f(x), \underline{d} f(z)) \le L \| z - x \|$. Therefore, by
Proposition~\ref{prp:PHdist} there exists $(a_0, x_0^*) \in \underline{d} f(x)$ such that 
$\| (0, x_z^*) - (a_0, x_0^*) \| \le L \| z - x \| + \varepsilon$. Hence
\begin{align*}
  f(y) &- f(x) - \max_{(a, x^*) \in \underline{d} f(x)} (a + \langle x^*, y - x \rangle)
  \\
  &= \langle x_z^*, y - x \rangle - \max_{(a, x^*) \in \underline{d} f(x)} (a + \langle x^*, y - x \rangle)
  \\
  &\le \langle x_z^*, y - x \rangle - a_0 - \langle x_0^*, y - x \rangle
  \le \big\| (0, x_z^*) - (a_0, x_0^*) \big\| \big( 1 + \| y - x \| \big)
  \\
  &\le \Big( L \| z - x \| + \varepsilon \Big) \big( 1 + \| y - x \| \big).
\end{align*}
It remains to pass to the limit as $\varepsilon \to + 0$ and note that $\| z - x \| \le \| y - x \|$, since 
$z \in \co\{ x, y \}$.
\end{proof}

Let us now turn to the analysis of existence of hypodifferentials of nonsmooth convex functions having the Lipschitz
property. We start with the simplest case of global hypodifferentials. Namely, let us show that global hypodifferentials
maps are, in a sense, always globally Lipschitz continuous (cf. examples in Subsection~\ref{subsect:Examples}). We need
an
auxiliary definition to precisely formulate this result.

\begin{definition}
Let a hypodifferential map $\underline{d} f$ of $f$ be defined on an open set $Q \subseteq \dom f$. This map is said to
be exact on $Q$, if
\[
  f(y) - f(x) = \max_{(a, x^*) \in \underline{d} f(x)} (a + \langle x^*, y - x \rangle)
  \quad \forall y, x \in Q.
\]
\end{definition}

\begin{proposition} \label{prp:GlobalHypodiff_LipContin}
If there exists a hypodifferential map of $f$ that is defined and exact on an open convex set 
$Q \subseteq \interior \dom f$, then $f$ is Lipschitz continuous on $Q$ and there exists an exact hypodifferential map
of $f$ that is defined and has the Lipschitz property on $Q$.
\end{proposition}

\begin{proof}
Let us define a new hypodifferential map $D_f(\cdot)$ of $f$ on $Q$. To this end, fix any point $x_0 \in X$ and choose
any $x \in Q$ and $\Delta x \in X$ such that $x + \Delta x \in Q$ (recall that $Q$ is an open set). By applying the
exactness property of the hypodifferential one gets
\begin{align*}
  f(x + \Delta x) - f(x) &= f(x_0) - f(x) 
  + \max_{(a, x^*) \in \underline{d} f(x_0)} (a + \langle x^*, x + \Delta x - x_0 \rangle)
  \\
  &= \max_{(a, x^*) \in D_f(x)} (a + \langle x^*, \Delta x \rangle),
\end{align*}
where
\[
  D_f(x) = \Big\{ (a + \langle x^*, x - x_0 \rangle + f(x_0) - f(x), x^*) \in \mathbb{R} \times X^* \Bigm|
  (a, x^*) \in \underline{d} f(x_0) \Big\}.
\]
Note that
\begin{align*}
  \max_{(a, x^*) \in D_f(x)} a 
  &= \max_{(a, x^*) \in \underline{d} f(x_0)} \big( a + \langle x^*, x - x_0 \rangle \big) + f(x_0) - f(x) 
  \\
  &= f(x) - f(x_0) + f(x_0) - f(x) = 0.
\end{align*}
Furthermore, the set $D_f(x)$ is convex and weak${}^*$ compact, since $D_f(x)$ is the image of the convex weak${}^*$
compact set $\underline{d} f(x_0)$ under the affine map
\[
  \mathcal{T}_x(a, x^*) = (a + \langle x^*, x - x_0 \rangle + f(x_0) - f(x), x^*)
\]
that is obviously continuous with respect to the weak${}^*$ topology. Thus, $D_f(x)$ is a hypodifferential of $f$ at
$x$. Clearly, $D_f(\cdot)$ is an exact hypodifferential map of $f$ on $Q$, which implies that $D_f(\cdot)$ is a
Lipschitzian approximation of $f$ on $Q$. Let us check that this hypodifferential map is also Lipschitz continuous 
on $Q$.

First, note that the quantity $L := \sup\{ \| x^* \| \mid (a, x^*) \in D_f(x) \}$ does not depend on $x$
by the definition of $D_f(\cdot)$. Therefore by Proposition~\ref{prp:BoundedHypodiff_LipContin} the function $f$ is
Lipschitz
continuous on $Q$ with Lipschitz constant $L$.

Fix any $x, y \in Q$. Let $(b, y^*) \in D_f(y)$. Then there exists 
$(a_0, x_0^*) \in \underline{d} f(x_0)$ such that $(b, y^*) = \mathcal{T}_y(a_0, x_0^*)$. Let 
$(a, x^*) = \mathcal{T}_x(a_0, x_0^*)$. Then $(a, x^*) \in D_f(x)$ and
\begin{align*}
  \dist( (b, y^*), D_f(x) ) \le \| (b, y^*) - (a, x^*) \| &\le | \langle x_0^*, y - x \rangle | + |f(y) - f(x)|
  \\
  &\le L \| y - x \| + L \| y - x \|.
\end{align*}
Consequently, one has
\[
  \sup_{(b, y^*) \in D_f(y)} \dist\big( (b, y^*), D_f(x) \big) \le 2 L \| y - x \|.
\]
Now, repeating the same argument with $x$ and $y$ swapped one can conclude that 
$d_{PH}(D_f(x), D_f(y)) \le 2 L \| y - x \|$. Thus, the hypodifferential map $D_f(\cdot)$ is Lipschitz continuous 
on $Q$.
\end{proof}

Nonetheless, as one can expect, not every convex function has a globally Lipschitz continuous hypodifferential mapping.

\begin{proposition}
Let $\dom f = X$ and there exist a hypodifferential map of $f$ defined on $X$ and globally Lipschitz continuous. Then
there exists $C_1, C_2 \ge 0$ such that $f(x) \le C_1 + C_2 \| x \|^2$ for all $x \in X$.
\end{proposition}

\begin{proof}
Fix any $x_0 \in X$ and denote $\theta = \sup\{ \| x^* \| \mid (a, x^*) \in \underline{d} f(x_0) \}$. Let also $L$ be a
Lipschitz constant from the definition of Lipschitz continuous hypodifferential map. By
Proposition~\ref{prp:LipContinHypodiff_Majorant} one has
\[
  f(x) \le f(x_0) + \max_{(a, x^*) \in \underline{d} f(x_0)} (a + \langle x^*, x - x_0 \rangle)
  + L \| x - x_0 \| \big( 1 + \| x - x_0 \| \big)
\]
for any $x \in X$. Taking into account the fact that $a \le 0$ for any $(a, x^*) \in \underline{d} f(x_0)$ and applying
the obvious inequality $(a + b)^2 \le 2 a^2 + 2 b^2$ one gets that
\begin{align*}
  f(x) &\le f(x_0) + (\theta + L) \| x - x_0 \| + L \| x - x_0 \|^2
  \\
  &\le f(x_0) + (\theta + L) \| x_0 \| + 2 L \| x_0 \|^2 + (\theta + L) \| x \| + 2 L \| x \|^2.
\end{align*}
for any $x \in X$. Hence with the use of the inequality $\| x \| \le 0.5 + 0.5 \| x \|^2$ one obtains that
$f(x) \le C_1 + C^2 \| x \|^2$ with
\[
  C_1 = f(x_0) + (\theta + L) \| x_0 \| + 2 L \| x_0 \|^2 + \frac{1}{2}(\theta + L), \quad
  C_2 = 2 L + \frac{1}{2}(\theta + L),
\]
and the proof is complete.
\end{proof}

\begin{remark}
Note that if there exists an exact hypodifferential map of $f$ on $X$, then with the use of the equality
\[
  f(x) - f(x_0) = \max_{(a, x^*) \in \underline{d} f(x_0)} (a + \langle x^*, x - x_0 \rangle) \quad \forall x \in X
\]
one gets that there exist $C_1, C_2 \ge 0$ such that $|f(x)| \le C_1 + C_2 \| x \|$ for all $x \in X$.
\end{remark}

Thus, for the convex function $f(x) = \| x \|^p$ with $p > 2$ there does not exist a globally Lipschitz continuous
hypodifferential map (or a hypodifferential map that has the Lipschitz property on the entire space). However, in the
case when the space $X$ is reflexive, a consistent hypodifferential map of a real-valued convex function that has 
the Lipschitz property locally \textit{always} exists. Moreover, under some mild assumptions there even exists a
consistent hypodifferential map that has the Lipschitz property on any given bounded set. The proof of this result
is inspired by the discussion from \cite[Appendix~I]{DemyanovRubinov}.

\begin{theorem} \label{thrm:HypodiffExistence}
Let $X$ be a relfexive Banach space, $Q \subseteq \interior \dom f$ be a bounded open set, and $f$ be Lipschitz
continuous on $Q$. Then there exists a hypodifferential mapping of $f$ that is defined, consistent, exact, and 
has the Lipschitz property on $Q$. In particular, the following statements hold true:
\begin{enumerate}
\item{for any $x \in \interior \dom f$ there exists a neighbourhood $\mathcal{U}_x$ of $x$ and a hypodifferential
mapping of $f$ that is defined, consistent, exact, and has the Lipschitz property on $\mathcal{U}_x$;}

\item{for any bounded open set $Q \subseteq \interior \dom f$ such that the subdifferential map $\partial f(\cdot)$ is
bounded on $Q$, there exists a hypodifferential mapping of $f$ that is defined, consistent, exact, and 
has the Lipschitz property on $Q$.
}
\end{enumerate}
\end{theorem}

\begin{proof}
For the sake of convenience, we divide the proof of the theorem into several parts. First, let us note that the
existence of a hypodifferential map having the required properties in a neighbourhood of any point 
$x \in \interior \dom f$ follows directly from \cite[Corollary~I.2.4]{EkelandTemam}. In turn, the last statement of the
theorem follows from the fact that if the subdifferential of $f$ is bounded on a set $Q$, then $f$ is Lipschitz
continuous on this set (see \cite[Theorem~2.4.13]{Zalinescu} and \cite[Theorem~24.7]{Rockafellar}). Let us now turn to 
the proof of the main statement of the theorem.

\textbf{1. The definition of hypodifferential.} By \cite[Corollary~I.2.5 and Proposition~I.5.2]{EkelandTemam} 
the function $f$ is subdifferentiable on $\interior \dom f$. Let $V(\cdot)$ be a selection of the restriction of 
the subdifferential mapping $\partial f(\cdot)$ to the set $Q$. By the definition of subdifferential
\begin{equation} \label{eq:SubgradientInequal}
  f(y) - f(x) \ge \langle V(x), y - x \rangle \quad \forall y, x \in Q.
\end{equation}
Note that this inequality turns into equality, if $y = x$. Therefore
\[
  f(y) = \max_{x \in Q} \big( f(x) + \langle V(x), y - x \rangle) \quad \forall y \in Q.
\]
Choose any $y \in Q$, $\Delta y \in X$ such that $y + \Delta y \in Q$. Then
\begin{equation} \label{eq:HypodiffDefViaSubgradIneq}
\begin{split}
  f(y + \Delta y) - f(y) &= \max_{x \in Q} \Big( f(x) - f(y) + \langle V(x), y + \Delta y - x \rangle \Big)
  \\
  &= \max_{(a, x^*) \in \underline{d} f(y)} (a + \langle x^*, \Delta y \rangle),
\end{split}
\end{equation}
where $\underline{d} f(y) = \cl^* D(y)$ and
\[
  D(y) = \co\Big\{ \big(f(x) - f(y) + \langle V(x), y - x \rangle, V(x) \big) \in \mathbb{R} \times X^* \Bigm|
  x \in Q \Big\}.
\]
Note that $\max\{ a \mid (a, x^*) \in D(y) \} = 0$ by inequality \eqref{eq:SubgradientInequal}, which obviously implies
that $\max\{ a \mid (a, x^*) \in \underline{d} f(y) \} = 0$. Let us check that the set $D(x)$ is bounded for any 
$x \in Q$. Then one can conclude that $\underline{d} f(y)$ is a weak${}^*$ compact set and, therefore, 
$\underline{d} f(y)$ is a hypodifferential of $f$ at $y$ that is obviously consistent and exact on $Q$ due to equality
\eqref{eq:HypodiffDefViaSubgradIneq}.

\textbf{2. Boundedness of hypodifferential.} Choose any $x \in Q$ and $(a, x^*) \in D(x)$. By definition one can find 
$n \in \mathbb{N}$, $z_i \in Q$, $i \in I := \{ 1, \ldots, n \}$, $V(z_i) \in \partial f(z_i)$, $i \in I$, 
and $\alpha_i \ge 0$, $i \in I$, such that
\begin{equation} \label{eq:ConstructiveExHypograd}
  (a, x^*) = \sum_{i = 1}^n \alpha_i \Big( f(z_i) - f(x) + \langle V(z_i), x - z_i \rangle, V(z_i) \Big), \quad
  \sum_{i = 1}^n \alpha_i = 1.
\end{equation}
By our assumptions the set $Q$ is bounded and $f$ is Lipschitz continuous on $Q$. Therefore, there exist $C_Q > 0$ and 
$C_f > 0$ such that $\| z \| \le C_Q$ and $|f(z)| \le C_f$ for all $z \in Q$. 

Let $L > 0$ be a Lipschitz constant of $f$ on $Q$. Then for any $z \in Q$ and $V \in \partial f(z)$ one has
\[
  \langle V, y - z \rangle \le f(y) - f(z) \le L \| y - z \| \quad \forall y \in Q.
\]
Since $Q$ is an open set, there exists $r > 0$ such that $B(z, r) \subset Q$. Hence with the use of the inequality
above one gets that
\[
  \langle V, y \rangle \le L \| y \| \quad \forall y \in B(0, r),
\]
which obviously implies that $\| V \| \le L$.

Bearing in mind equality \eqref{eq:ConstructiveExHypograd} one has
\begin{align*}
  \| (a, x^*) \| 
  &\le \sum_{i = 1}^n \alpha_i \Big( \big| f(z_i) - f(x) + \langle V(z_i), x - z_i \rangle \big| + \| V(z_i) \| \Big)
  \\	
  &\le \sum_{i = 1}^n \alpha_i \Big( 2 C_f + 2 L C_Q + L \Big) = 2 C_f + 2 L C_Q + L
\end{align*}
for any $(a, x^*) \in D(x)$. Thus, the set $D(x)$ is bounded.

\textbf{3. Lipschitz continuity of hypodifferential.} As was noted in the proof of
Proposition~\ref{prp:WeakStarClosureContinuous}, the reflexivity of $X$ implies that the set 
$\underline{d} f(x) = \cl^* D(x)$ coincides with the closure $\cl D(x)$ of $D(x)$ in the strong (norm) topology.
Moreover, one obviously has $d_{PH}(\cl D(x), \cl D(y)) = d_{PH}(D(x), D(y))$ for any $x, y \in Q$, which implies that
it is sufficient to show that the set-valued map $D(\cdot)$ is Lipschitz continuous on $Q$.

Fix any $x, y \in Q$. Let $(a, x^*) \in D(x)$. Then by definition there exist $n \in \mathbb{N}$, $z_i \in Q$,
$i \in I := \{ 1, \ldots, n \}$, $V(z_i) \in \partial f(z_i)$, $i \in I$, and $\alpha_i \ge 0$, $i \in I$, satisfying
\eqref{eq:ConstructiveExHypograd}. Denote
\[
  (b, y^*) = \sum_{i = 1}^n \alpha_i \Big( f(z_i) - f(y) + \langle V(z_i), y - z_i \rangle, V(z_i) \Big).
\]
Clearly, $(b, y^*) \in D(y)$. Therefore
\begin{align*}
  \dist( (a, x^*), D(y) ) &\le \| (a, x^*) - (b, y^*) \| 
  = \Big| \sum_{i = 1}^n \alpha_i(f(x) - f(y) + \langle V(z_i), x - y \rangle \Big|
  \\
  &= \Big| f(x) - f(y) + \langle \widehat{V}, x - y \rangle \Big| \le |f(x) - f(y)| + \| \widehat{V} \| \| x - y \|,
\end{align*}
where $\widehat{V} = \sum_{i \in I} \alpha_i V(z_i)$. Let $L > 0$ be a Lipschitz constant of $f$ on $Q$. Then, as was
noted above, $\| V \| \le L$ for any $V \in \partial f(z)$ and $z \in Q$. Consequently, one has
\[
  \dist( (a, x^*), D(y) ) \le 2 L \| x - y \| \quad \forall (a, x^*) \in D(x).
\]
Taking the supremum over all $(a, x^*) \in D(x)$ and then repeating the same argument with $x$ and $y$ swapped one
obtains that $d_{PH}( D(x), D(y) ) \le 2 L \| x - y \|$ for any $x, y \in Q$.
\end{proof}

\begin{corollary} \label{crlr:ExistenceThrm_FiniteDim}
Let $X$ be finite dimensional. Then for any bounded set $Q \subseteq \interior \dom f$ there exists a hypodifferential
mapping of $f$ that is defined, consistent, exact, and has the Lipschitz property on $Q$.
\end{corollary}

\begin{corollary}
The function $f$ has an exact hypodifferential mapping on an open bounded set $Q \subseteq \interior \dom f$ if and only
if $f$ is Lipschitz continuous on this set.
\end{corollary}

Although Theorem~\ref{thrm:HypodiffExistence} guarantees the existence of a consistent hypodifferential having the
Lipschitz property under very mild assumptions, the hypodifferential constructed in the proof of this theorem is of
theoretical value only. From the more practical perspective it is important to know how to compute such
hypodifferentials for particular nonsmooth convex functions. Let us show that the calculus rules presented in
Subsection~\ref{subsect:CalculusRules} preserve the Lipschitz property and, therefore, can be used to
compute consistent hypodifferentials having the Lipschitz property in various particular cases.

\begin{theorem}
Under the assumptions of Theorem~\ref{thrm:MaxInfFamily} the following statements hold true:
\begin{enumerate}
\item{If for any $t \in T$ the hypodifferential $\underline{d} f_t$ is a Lipschitzian approximation of $f(\cdot, t)$ on
$Q$ with Lipschitz constant $L_t > 0$ and $L = \sup_{t \in T} L_t < + \infty$, then hypodifferential
\eqref{eq:InfMaxHypodiff} is a Lipschitzian approximation of $f(\cdot) = \sup_{t \in T} f(\cdot, t)$ on $Q$ with
Lipschitz constant $L$.
}

\item{Let either $X$ be reflexive or the set $\cl D(x)$ defined in \eqref{eq:InfMaxPreHypodiff} be weak${}^*$ closed for
any $x \in Q$. Suppose that for any $t \in T$ the function $f(\cdot, t)$ is Lipschitz continuous on $Q$ with Lipschitz
constant $L_t > 0$ and the hypodifferential map $\underline{d} f_t$ is Lipschitz continuous on $Q$ with Lipschitz
constant $K_t > 0$. If $L = \sup_{t \in T} L_t < + \infty$ and $K = \sup_{t \in T} K_t < + \infty$, then
hypodifferential \eqref{eq:InfMaxHypodiff} is Lipschitz continuous on $Q$ with Lipschitz constant $2 L + K$.
}
\end{enumerate}
\end{theorem}

\begin{proof}
We split the proof of the theorem into two parts corresponding to the two statements.

\textbf{Part 1.} Fix any $x, y \in Q$. By Definition~\ref{def:LipApprox} one has
\[
  f(y, t) \le f(x, t) + \max_{(a, x^*) \in \underline{d} f_t(x)} (a + \langle x^*, y - x \rangle)
  + \frac{L_t}{2} \| y - x \|^2.
\]
Hence
\begin{align*}
  f(y) &- f(x) = \sup_{t \in T} \big( f(y, t) - f(x) \big)
  \\
  &\le \sup_{t \in T} \Big( f(x, t) - f(x) + \max_{(a, x^*) \in \underline{d} f_t(x)} (a + \langle x^*, y - x \rangle)
  + \frac{L_t}{2} \| y - x \|^2 \Big)
  \\
  &\le \sup_{t \in T} \Big( f(x, t) - f(x) 
  + \max_{(a, x^*) \in \underline{d} f_t(x)} (a + \langle x^*, y - x \rangle \Big)
  + \frac{L}{2} \|y - x \|^2.
\end{align*}
One can readily verify that
\begin{multline*}
  \sup_{t \in T} \Big( f(x, t) - f(x) + \max_{(a, x^*) \in \underline{d} f_t(x)} (a + \langle x^*, y - x \rangle \Big)
  \\
  = \max_{(a, x^*) \in \underline{d} f(x)} (a + \langle x^*, y - x \rangle)
\end{multline*}
(see \eqref{eq:InfMaxHypodiff}). Therefore
\[
  f(y) - f(x) - \max_{(a, x^*) \in \underline{d} f(x)} (a + \langle x^*, y - x \rangle) \le \frac{L}{2} \|y - x \|^2.
\]
Applying the inequality
\[
  f(y, t) \ge f(x, t) + \max_{(a, x^*) \in \underline{d} f_t(x)} (a + \langle x^*, y - x \rangle)
  - \frac{L_t}{2} \| y - x \|^2,
\]
that also follows from Definition~\ref{def:LipApprox}, and repeating the same argument one gets that
\[
  f(y) - f(x) - \max_{(a, x^*) \in \underline{d} f(x)} (a + \langle x^*, y - x \rangle) \ge - \frac{L}{2} \|y - x \|^2.
\]
Consequently, $\underline{d} f$ is a Lipschitzian approximation of $f$ on $Q$.

\textbf{Part 2.} If either $X$ is reflexive or $\cl D(x)$ is weak${}^*$ closed for any $x \in X$, then, as was noted
several times above, $d_{PH}(\underline{d} f(x), \underline{d} f(y)) = d_{PH}(D(x), D(y))$ for any $x, y \in Q$.
Therefore, it is sufficient to prove that the set-valued map $D(\cdot)$ is Lipschitz continuous on $Q$.

Fix any $x, y \in Q$ and $\varepsilon > 0$. Let $(a, x^*) \in D(x)$. Then by definition (see
\eqref{eq:InfMaxPreHypodiff}) there exist $n \in \mathbb{N}$, $t_i \in T$, $i \in I := \{ 1, \ldots, n \}$, 
$(a_i, x_i^*) \in \underline{d} f_{t_i}(x)$, $i \in I$, and $\alpha_i \ge 0$, $i \in I$, such that
\[
  (a, x^*) = \sum_{i = 1}^n \alpha_i (f(x, t_i) - f(x) + a_i, x_i^*), \quad
  \sum_{i = 1}^n \alpha_i = 1.
\]
By our assumption $d_{PH}(\underline{d} f_t(x), \underline{d} f_t(y)) \le K_t \| x - y \|$. Hence by
Proposition~\ref{prp:PHdist} there exist $(b_i, y_i^*) \in \underline{d} f_{t_i}(y)$, $i \in I$ such that 
$\| (a_i, x_i^*) - (b_i, y_i^*) \| \le K_{t_i} \| x - y \| + \varepsilon$. Define
\[
  (b, y^*) = \sum_{i = 1}^n \alpha_i (f(y, t_i) - f(y) + b_i, y_i^*).
\]
Clearly, $(b, y^*) \in D(y)$. Therefore
\begin{multline*}
  \dist( (a, x^*), D(y) ) \le \| (a, x^*) - (b, y^*) \| 
  \\
  \le \sum_{i = 1}^n \alpha_i \Big( |f(x, t_i) - f(y, t_i)| + |f(x) - f(y)| + \| (a_i, x_i^*) - (b_i, y_i^*) \| \Big)
  \\
  \le \sum_{i = 1}^n \alpha_i \Big( L_t \| x - y \| + L \| x - y \| + K_{t_i} \| x - y \| + \varepsilon \Big)
  \le (2 L + K) \| x - y \| + \varepsilon 
\end{multline*}
(here we used the fact that the function $f(\cdot) = \sup_{t \in T} f(\cdot, t)$ is Lipschitz continuous on $Q$ with
Lipschitz constant $L = \sup_{t \in T} L_t$). Since $(a, x^*) \in D(x)$ and $\varepsilon > 0$ were chosen arbitrarily,
one obtains
\[
  \sup_{(a, x^*) \in D(x)} \dist( (a, x^*), D(y) ) \le (2 L + K) \| x - y \|.
\]
Swapping $x$ and $y$ one can conclude that $d_{PH}(D(x), D(y)) \le (2 L + K) \| x - y \|$, which completes the proof.
\end{proof}

\begin{theorem}
Let the assumptions of Theorem~\ref{thm:OuterComposition} be satisfied, the gradient $\nabla g(\cdot)$ of $g$ be defined
and Lipschitz continuous with Lipschitz constant $L_g > 0$ on the set $\co f(Q)$, and there exist $C > 0$ such that 
$\| \nabla g(f(x)) \| \le C$ for any $x \in Q$. Then the following statements hold true:
\begin{enumerate}
\item{If for any $i \in I$ the hypodifferential map $\underline{d} f_i$ is a Lipschitzian approximation of $f_i$ on $Q$
with Lipschitz constant $L_i$ and $f_i$ is Lipschitz continuous on $Q$ with Lipschitz constant $K_i$, then 
the hypodifferential map $\underline{d} h$ (see \eqref{eq:OuterCompositionHypodiff}) is a Lipschitzian approximation of
the function $h(\cdot) = g(f(\cdot))$ with Lipschitz constant $L = C \sum_{i \in I} L_i + 2 L_g (\sum_{i \in I} K_i)^2$.
}

\item{Suppose that for any $i \in I$ the hypodifferential map $\underline{d} f_i$ is Lipschitz continuous on $Q$ with
Lipschitz constant $L_i > 0$ and there exists $C_i > 0$ such that $\| (a, x^*) \| \le C_i$ for all 
$(a, x^*) \in \underline{d} f_i(x)$ and $x \in Q$. Then the hypodifferential map $\underline{d} h$ is Lipschitz
continuous on $Q$ with Lipschitz constant $l = C \sum_{i \in I} L_i + L_g \sum_{i \in I} C_i$.
}
\end{enumerate}
\end{theorem}

\begin{proof}
We split the proof of the theorem into two parts corresponding to the two statements.

\textbf{Part 1.} Fix any $x, y \in Q$. By the mean value theorem there exists $z \in \co\{ f(x), f(y) \}$ such that
\begin{align*}
  h(y) - h(x) &= g(f(y)) - g(f(x)) = \langle \nabla g(z), f(y) - f(x) \rangle 
  \\
  &= \langle \nabla g(f(x)), f(y) - f(x) \rangle 
  + \langle \nabla g(z) - \nabla g(f(x)), f(y) - f(x) \rangle.
\end{align*}
Hence applying the fact that $\underline{d} f_i$ is a Lipschitzian approximation of $f_i$ one gets
\begin{align*}
  h(y) - h(x) &\le \sum_{i = 1}^n \frac{\partial g(f(x))}{\partial y_i} 
  \Big( \max_{(a_i, x_i^*) \in \underline{d} f_i(x)} (a_i + \langle x_i^*, y - x \rangle) 
  \\
  &+ \frac{L_i}{2} \| y - x \|^2 \Big)
  + \langle \nabla g(z) - \nabla g(f(x)), f(y) - f(x) \rangle
\end{align*}
(note that $\nabla g(f(x)) \ge 0$, since $g$ is a non-decreasing function by the assumptions of
Theorem~\ref{thm:OuterComposition}). Bearing in mind the assumptions of the theorem and the definition of 
$\underline{d} h(x)$ (see \eqref{eq:OuterCompositionHypodiff})), one obtains
\begin{align*}
  h(y) - h(x) &\le \max_{(a, x^*) \in \underline{d} h(x)} (a + \langle x^*, y - x \rangle) 
  + C \Big( \sum_{i = 1}^n \frac{L_i}{2} \Big) \| y - x \|^2 
  \\
  &+ L_g |z - f(x)| \cdot |f(y) - f(x)| \le \max_{(a, x^*) \in \underline{d} h(x)} (a + \langle x^*, y - x \rangle)
  \\
  &+ C \sum_{i = 1}^n \frac{L_i}{2} \| y - x \|^2 + L_g \Big( \sum_{i = 1}^n K_i \Big)^2 \| y - x \|^2
\end{align*}
(here we used the fact that $\| z - f(x) \| \le \| f(y) - f(x) \|$, since $z \in \co\{ f(x), f(y) \}$). Arguing in 
the same way and one can readily check that 
\[
  h(y) - h(x) \ge \max_{(a, x^*) \in \underline{d} h(x)} (a + \langle x^*, y - x \rangle) 
  - \Big( C \sum_{i = 1}^n \frac{L_i}{2} +  L_g (\sum_{i = 1}^n K_i)^2 \Big) \| y - x \|^2,
\]
which implies the required result.

\textbf{Part 2.} Fix any $x, y \in Q$. Let $(a, x^*) \in \underline{d} h(x)$. Then by definition (see
\eqref{eq:OuterCompositionHypodiff}) for any $i \in I$ there exists $(a_i, x_i^*) \in \underline{d} f_i(x)$ such that
\[
  (a, x^*) = \sum_{i = 1}^n \frac{\partial g(f(x))}{\partial y_i} (a_i, x_i^*).
\]
Choose any $\varepsilon > 0$. By our assumption 
$d_{PH}(\underline{d} f_i(x), \underline{d} f_i(y)) \le L_i \| x - y \|$. Therefore,
by Proposition~\ref{prp:PHdist} for any $i \in I$ there exists $(b_i, y_i^*) \in \underline{d} f_i(y)$ such that
$\| (a_i, x_i^*) - (b_i, y_i^*) \| \le L_i \| x - y \| + \varepsilon$. Define
\[
  (b, y^*) = \sum_{i = 1}^n \frac{\partial g(f(y))}{\partial y_i} (b_i, y_i^*).
\]
Clearly, $(b, y^*) \in \underline{d} h(y)$, which implies that
\begin{align*}
  \dist( (a, x^*), \underline{d} h(y) ) &\le \| (a, x^*) - (b, y^*) \|
  \\
  &= \Big\| \sum_{i = 1}^n \frac{\partial g(f(x))}{\partial y_i} (a_i, x_i^*) 
  - \sum_{i = 1}^n \frac{\partial g(f(y))}{\partial y_i} (b_i, y_i^*) \Big\|.
\end{align*}
Adding and subtracting $\sum_{i \in I} \frac{\partial g(f(y))}{\partial y_i} (a_i, x_i^*)$ one gets
\begin{align*}
  \dist( (a, x^*), \underline{d} h(y) ) &\le \sum_{i = 1}^n 
  \left| \frac{\partial g(f(x))}{\partial y_i} -  \frac{\partial g(f(y))}{\partial y_i}  \right| \| (a_i, x_i^*) \|
  \\
  &+ \sum_{i = 1}^n \left| \frac{\partial g(f(y))}{\partial y_i} \right| \| (a_i, x_i^*) - (b_i, y_i^*) \|
  \\
  &\le L_g \Big( \sum_{i = 1}^n C_i \Big) \| x - y \| + \sum_{i = 1}^n C (L_i \| x - y \| + \varepsilon).
\end{align*}
Hence taking into account the facts that $(a, x^*) \in \underline{d} h(x)$ and $\varepsilon > 0$ were chosen
arbitrarily one obtains
\[
  \sup_{(a, x^*) \in \underline{d} h(x)} \dist( (a, x^*), \underline{d} h(y) )
  \le \Big( L_g \sum_{i = 1}^n C_i + C \sum_{i = 1}^n L_i \Big) \| x - y \|.
\]
Now, swapping $x$ and $y$ we arrive at the required result.
\end{proof}

\begin{corollary}
Let $f_i \colon X \to \mathbb{R} \cup \{ + \infty \}$, $i \in I = \{ 1, \ldots, n \}$, be closed convex functions and
$\underline{d} f_i(\cdot)$ be their hypodifferential mappings defined on a set 
$Q \subseteq \bigcap_{i \in I} \dom f_i$. Let also $\lambda_i \ge 0$, $i \in I$, $f = \sum_{i \in I} \lambda_i f_i$, and
$\underline{d} f = \sum_{i \in I} \lambda_i d f_i$. The following statements hold true:
\begin{enumerate}
\item{If for any $i \in I$ the hypodifferential map $\underline{d} f_i$ is a Lipschitzian approximation of $f_i$ on $Q$
with Lipschitz constant $L_i > 0$, then the hypodifferential map $\underline{d} f$ is a Lipschitzian approximation of
$f$ on $Q$ with Lipschitz constant $L = \sum_{i \in I} \lambda_i L_i$.
}

\item{If for any $i \in I$ the hypodifferential map $\underline{d} f_i$ is Lipschitz continuous on $Q$ with Lipschitz
constant $L_i > 0$, then the hypodifferential map $\underline{d} f$ is Lipschitz continuous on $Q$ with Lipschitz
constant $L = \sum_{i \in I} \lambda_i L_i$.
}
\end{enumerate}
\end{corollary}

The following result can be readily verified directly.

\begin{proposition}
Under the assumptions of Proposition~\ref{prp:AffineShiftHypodiff} the following statements hold true:
\begin{enumerate}
\item{If the hypodifferential $\underline{d} f$ is a Lipschitzian approximation of $f$ on $Q$ with Lipschitz constant
$L$, then the hypodifferential $\underline{d} g$ (see \eqref{eq:AffineShiftHypodiff}) is a Lipschitzian approximation of
$g$ on $A^{-1}(Q - b)$ with Lipschitz constant $L \| A \|$.
}

\item{If the hypodifferential map $\underline{d} f$ is Lipschitz continuous on $Q$ with Lipschitz constant $L$, then 
the hypodifferential $\underline{d} g$ is Lilpschitz continuous on $A^{-1}(Q - b)$ with Lipschitz constant 
$L \max\{ 1, \| A^* \| \}$.
}
\end{enumerate}
\end{proposition}

\section{Applications to nonsmooth convex optimization}
\label{sect:Applications}

In this section we consider applications of the theory of hypodifferentials of nonsmooth convex function to convex
optimization. Our goal is to present several different versions of the method of hypodifferential descent for nonsmooth
convex minimization and estimate their rate of convergence.

\subsection{The method of hypodifferential descent}
\label{subsect:MHD}

For the sake of simplicity, suppose that $X$ is a real Hilbert space, although under some additional assumptions 
the results presented below can be partially extended to the case of general Banach spaces. 
Let the space $\mathbb{R} \times X$ be equipped with the norm $\| (a, x) \| = \sqrt{|a|^2 + \| x \|^2}$ for any 
$(a, x) \in \mathbb{R} \times X$ and $f \colon X \to \mathbb{R}$ be a closed convex function. Suppose that 
a hypodifferential map $\underline{d} f(\cdot)$ of this function defined on $X$ is known. Recall that in the case of
Hilbert spaces hypodifferential is a convex weakly compact (or, equivalently, convex, bounded, and closed) subset of 
the space $\mathbb{R} \times X$ (see Remark~\ref{rmrk:HypodiffDef}).

Consider the following optimization problem:
\[
  \min_{x \in X} \enspace f(x).
\]
Let $f_*$ be the optimal value of the problem. Following the ideas of Demyanov \cite{DemyanovRubinov} one can propose a
method for minimizing the function $f$ called \textit{the method of hypodifferential descent}, whose theoretical
scheme is given in Algorithm~\ref{alg:HypodiffDescent}.

\begin{algorithm} \label{alg:HypodiffDescent}
\caption{The method of hypodifferential descent.}

\textbf{Initialisation.} Choose an initial point $x_0 \in X$ and put $k = 0$.

\textbf{Step 1.} Compute $\underline{d} f(x_k)$. Find an optimal solution $(a_k, v_k)$ of the problem
\begin{equation} \label{prob:DistToHypodiff}
  \min \enspace a^2 + \| v \|^2 \quad \text{ subject to } (a, v) \in \underline{d} f(x_k).
\end{equation}

\textbf{Step 2.} Choose a step size $\alpha_k > 0$ and put $x_{k + 1} = x_k - \alpha_k v_k$.

\textbf{Step 3.} Check a \textbf{stopping criterion}. If it is satisfied, \textbf{Stop}. Otherwise, 
put $k \leftarrow k + 1$ and go to \textbf{Step 1}.
\end{algorithm}

Note that if the hypodifferential $\underline{d} f(x_k)$ is a polytope, then the subproblem on Step~2 of the method of
hypodifferential descent is the problem of finding an element of a polytope with the smallest norm for which there
exists many efficient methods \cite{Gilbert,Wolfe,BarberoLopez,Dolgopolik_AccelTechnique}. It should also be mentioned
that one can use the following inequalities
\[
  |f(x_{k + 1}) - f(x_k)| \le \varepsilon \quad \text{and/or} \quad
  \dist\big( 0, \underline{d} f(x_k) \big) = \sqrt{a_k^2 + \| v_k \|^2} \le \varepsilon
\]
with small $\varepsilon > 0$ as a stopping criterion for the method of hypodifferential descent.

The rate of convergence $\mathcal{O}(1/k)$ of the method of hypodifferential with the step sizes $\alpha_k$ chosen
according to the Armijo's rule was proved in \cite{Dolgopolik_HypodiffDescent}. Here we will prove that the same
rate of convergence holds true for the version of the method with constant step size and provide a rigorous
justification for the stopping criteria discussed above. Let us note that although the proof of this result is similar
to the proof of \cite[Theorem~3.5]{Dolgopolik_HypodiffDescent}, there are certain significant differences between them
(e.g. the fact that one needs find an upper estimate for the constant step size ensuring that the sequence 
$\{ f(x_k) \}$ is decreasing).

\begin{theorem} \label{thm:MHD_ConvergenceRate}
Let the sublevel set $S(x_0) = \{ x \in X \mid f(x) \le f(x_0) \}$ be bounded. Suppose also that the hypodifferential
mapping $\underline{d} f(\cdot)$ is consistent and bounded on the set $S(x_0)$ and there exists $\varepsilon > 0$ such
that $\underline{d} f$ is a Lipschitzian approximation of $f$ on the set 
$S_{\varepsilon}(x_0) = \{ x \in X \mid \dist(x, S(x_0)) \le \varepsilon \}$ with Lipschitz constant $L > 0$. Then there
exists $\widehat{\alpha}$ such that for any $\alpha \in (0, \widehat{\alpha})$ the method of hypodifferential descent
with $\alpha_k \equiv \alpha$ generates sequence $\{ x_k \}$ satisfying the following conditions:
\begin{enumerate}
\item{the sequence $\{ f(x_k) \}$ is strictly decreasing and converges to the global minimum of $f$ (in particular,
$|f(x_{k + 1}) - f(x_k)| \to 0$ as $k \to \infty$);}

\item{$\sum_{i = 0}^{\infty} \| (a_k, v_k) \|^2 < + \infty$ and $\dist(0, \underline{d} f(x_k)) \to 0$ as 
$k \to \infty$;}

\item{for all $k \in \mathbb{N}$ the following inequality holds true:
\begin{equation} \label{eq:MHD_RateOfConvergence}
  f(x_k) - f_* 
  \le \frac{(f(x_0) - f_*)(1 + R)^2}{(1 + R)^2 + (f(x_0) - f_*) \left( \alpha - \frac{L \alpha^2}{2} \right) k}
  = \mathcal{O}\left(\frac{1}{k}\right)
\end{equation}
where $R = \sup\{ \| y - x_* \| \mid y \in S(x_0) \}$, $x_*$ is any point of global minimum of $f$, and $f_* = f(x_*)$.
}
\end{enumerate}
Moreover, one can define $\widehat{\alpha} = \min\{ 2/L, \min\{ 1, \varepsilon \}/C \}$, where $C > 0$ is such that 
$|a| \le C$ and $\| v \| \le C$ for any $(a, v) \in \underline{d} f(x)$ and $x \in S(x_0)$.
\end{theorem}

\begin{proof}
Note that under the assumptions of the theorem the function $f$ attains a global minimum, since $X$ is a Hilbert space
and the sublevel set $S(x_0)$ is bounded.

Let $\{ x_k \}$ be the sequence generated by the method of hypodifferential descent with $\alpha_k \equiv \alpha$ for
some $\alpha > 0$. Denote $\Phi_k(y) = \max_{(a, v) \in \underline{d} f(x_k)} (a + \langle v, y \rangle)$. 
For the sake of convenience, we divide the proof of the theorem into several parts.

\textbf{Part 1.} Recall that $(a_k, v_k)$ is an optimal solution of subproblem \eqref{prob:DistToHypodiff}. Applying 
the standard optimality conditions to this subproblem one gets
\begin{equation} \label{eq:HypodiffDist_OptCond}
  a_k (a - a_k) + \langle v_k, v - v_k \rangle \ge 0 \quad \forall (a, v) \in \underline{d} f(x_k).
\end{equation}
If $a_k = 0$, then bearing in mind the fact that $a \le 0$ for any $(a, v) \in \underline{d} f(x_k)$ one obtains
\[
  \Phi_k(- v_k) \le \max_{(a, v) \in \underline{d} f(x_k)} \langle v, - v_k \rangle \le - \| v_k \|^2
  = - \| (a_k, v_k) \|^2.
\]
In turn, if $a_k < 0$, then dividing \eqref{eq:HypodiffDist_OptCond} by $a_k$ and taking the maximum over all 
$(a, v) \in \underline{d} f(x_k)$ one gets 
\[
  \Phi_k\left( \frac{1}{a_k} v_k \right) 
  = \max_{(a, v) \in \underline{d} f(x_k)} \left( a + \frac{1}{a_k} \langle v, v_k \rangle \right)
  \le \frac{1}{a_k} \| (a_k, v_k) \|^2.
\]
Thus, in either case one has $\Phi_k(- \beta_k v_k) \le - \beta_k \| (a_k, v_k) \|^2$, where $\beta_k = 1 / |a_k|$, if
$a_k < 0$, and $\beta_k = 1$, otherwise. Hence taking into account the fact that by the definition of hypodifferential
$\Phi_k(0) = 0$ one obtains that
\[
  \Phi_k(- t \beta_k v_k) \le t \Phi_k( - \beta_k v_k) \le - t \beta_k \| (a_k, v_k) \|^2
  \quad \forall t \in [0, 1], 
\]
thanks to the convexity of $\Phi_k$, or, equivalently,
\begin{equation} \label{eq:MHD_MainIneq}
  \Phi_k(- t v_k) \le - t \| (a_k, v_k) \|^2 \quad \forall t \in [0, \beta_k]
\end{equation}
for any $k \in \mathbb{N}$.

\textbf{Part 2.} Let us show that there exists $\widehat{\alpha} > 0$ such that for any 
$\alpha \in (0, \widehat{\alpha}]$ one has $\{ x_k \} \subset S(x_0)$ and the sequence $\{ f(x_k) \}$ is strictly
decreasing. Indeed, by our assumption the hypodifferential map $\underline{d} f(\cdot)$ is bounded on $S(x_0)$.
Therefore, there exists $C > 0$ such that $|a| \le C$ and $\| v \| \le C$ for any $(a, v) \in \underline{d} f(x)$ and 
$x \in S(x_0)$. Define
\begin{equation} \label{eq:StepSizeUpperBound}
  \widehat{\alpha} = \min\left\{ \frac{2}{L}, \frac{\min\{ 1, \varepsilon \}}{C} \right\}
\end{equation}
and let $\alpha \in (0, \widehat{\alpha})$.

Suppose that $x_k \in S(x_0)$ for some $k \in \mathbb{N}$. Let us check that $x_{k + 1} \in S(x_0)$ and 
$f(x_{k + 1}) < f(x_k)$. Indeed, since $x_k \in S(x_0)$, one has $|a_k| \le C$ or, equivalently, $\beta_k \ge 1 / C$
(one can suppose that $C \ge 1$), which implies that $\alpha_k = \alpha \le \beta_k$. Moreover, from 
the definition of $\widehat{\alpha}$ it follows that $t \| v_k \| \le \varepsilon$ for any 
$t \in [0, \widehat{\alpha}]$. Consequently, $x_{k + 1} := x_k - \alpha v_k \in S_{\varepsilon}(x_0)$. Therefore,
applying inequality \eqref{eq:MHD_MainIneq} and the fact that the hypodiferential map $\underline{d} f(\cdot)$ is a
Lipschitzian approximation of $f$ on $S_{\varepsilon}(x_0)$ one gets that
\begin{multline*}
  f(x_{k + 1}) \le f(x_k) + \max_{(a, v) \in \underline{d} f(x_k)} (a + \langle v, x_{k + 1} - x_k \rangle)
  + \frac{L}{2} \| x_{k + 1} - x_k \|^2
  \\
   \le f(x_k) - \alpha \| (a_k, v_k) \|^2 + \frac{L \alpha^2}{2} \| v_k \|^2
  \le f(x_k) - \left( \alpha - \frac{L \alpha^2}{2} \right) \| (a_k, v_k) \|^2.
\end{multline*}
Hence taking into account the fact that $\alpha < 2 / L$ (see \eqref{eq:StepSizeUpperBound}) one can conclude that
$f(x_{k + 1}) < f(x_k)$ and $x_{k + 1} \in S(x_0)$. Thus, in particular, the sequence $\{ f(x_k) \}$ converges.

\textbf{Part 3.} Denote $\Delta_k = f(x_k) - f_*$. As we have shown above, one has
\begin{equation} \label{eq:MHD_DeltaDecay}
  \Delta_{k + 1} - \Delta_k \le - \eta \| (a_k, v_k) \|^2 \quad \forall k \in \mathbb{N},
  \quad \eta = \alpha - \frac{L \alpha^2}{2} > 0.
\end{equation}
Hence, in particular, $\Delta_{k + 1} < \Delta_k$ and
\[
  f(x_{k + 1}) - f(x_0) = \sum_{i = 0}^k (\Delta_{i + 1} - \Delta_i)
  \le - \eta \sum_{i = 0}^k \| (a_k, v_k) \|^2 \quad \forall k \in \mathbb{N},
\]
which implies that
\[
  \sum_{i = 0}^{\infty} \| (a_k, v_k) \|^2 \le \frac{1}{\eta} \Big( f(x_0) - \lim_{k \to \infty} f(x_k) \Big)
  \le \frac{1}{\eta} (f(x_0) - f_*).
\]
Thus, the second statement of the theorem holds true.

Recall that the hypodifferential map $\underline{d} f(\cdot)$ is consistent and $(a_k, v_k) \in \underline{d} f(x_k)$.
Therefore
\begin{equation} \label{eq:Delta_UpperEstim}
  \Delta_k \le - a_k + \langle v_k, x_k - x_* \rangle \le \| (a_k, v_k) \| \big( 1 + \| x_k - x_* \| )
  \le (1 + R) \| (a_k, v_k) \|  
\end{equation}
for any $k \in \mathbb{N}$. Combining this inequality with \eqref{eq:MHD_DeltaDecay} one obtains
\[
  \Delta_{k + 1} \le \Delta_k - \frac{\eta}{(1 + R)^2} \Delta_k^2 \quad \forall k \in \mathbb{N}.
\]
Dividing this inequality by $\Delta_{k + 1} \cdot \Delta_k$ and rearranging the terms one gets
\[
  \frac{1}{\Delta_{k + 1}} \ge \frac{1}{\Delta_k} + \frac{\eta}{(1 + R)^2} \frac{\Delta_k}{\Delta_{k + 1}}
  \ge \frac{1}{\Delta_k} + \frac{\eta}{(1 + R)^2} \quad \forall k \in \mathbb{N}
\]
(here we used the fact that $\Delta_{k + 1} < \Delta_k$, that is, $\frac{\Delta_k}{\Delta_{k + 1}} > 1$). Summing up
these inequalities one finally obtains
\[
  \frac{1}{\Delta_{k + 1}} \ge \frac{1}{\Delta_0} + \frac{\eta}{(1 + R)^2} (k + 1) \quad \forall k \in \mathbb{N},
\]
which implies that the estimate \eqref{eq:MHD_RateOfConvergence} holds true.
\end{proof}

\begin{remark}
Note that the assumptions of the previous theorem are not restrictive in the finite dimensional case, since by
Corollary~\ref{crlr:ExistenceThrm_FiniteDim} a hypodifferential mapping satisfying these assumptions exists for any
nonsmooth function $f \colon X \to \mathbb{R}$ having the bounded sublevel set $S(x_0)$.
\end{remark}

Let us show that if the hypodifferential map $\underline{d} f(\cdot)$ is \textit{exact} (global), then one can use 
the first component $a_k$ of an optimal solution of subproblem~\eqref{prob:DistToHypodiff} on Step~1 of the method of
hypodifferential descent to determine a step size that ensures the linear rate of convergence of the method.

\begin{theorem} \label{thm:MHD_ExactCase}
Let the sublevel set $S(x_0)$ be bounded, the hypodifferential map $\underline{d} f(\cdot)$ be exact on
$S_{\varepsilon}(x_0)$ for some $\varepsilon > 0$, and $\{ x_k \}$ be the sequence generated by the method of
hypodifferential descent with $\alpha_k = 1 / a_k$, if $a_k < 0$, and $\alpha_k = 1$, otherwise. Then the following
statements hold true:
\begin{enumerate}
\item{$a_k < 0$ for any $k \in \mathbb{N}$;}

\item{there exists $\gamma \in (0, 1)$ and $C > 0$ such that $\dist(0, \underline{d} f(x_k)) \le C \gamma^k$
for all $k \in \mathbb{N}$;
}

\item{there exists $q \in (0, 1)$ such that $f(x_k) - f_* \le (f(x_0) - f_*)) q^k$ for any $k \in \mathbb{N}$.}
\end{enumerate}
\end{theorem}

\begin{proof}
We use the same notation as in the proof of Theorem~\ref{thm:MHD_ConvergenceRate}. For the sake of convenience, we
divide the proof of this theorem into several parts.

\textbf{Part 1.} Let us show that if $x_k \in S(x_0)$, then $a_k < 0$. Suppose by contradiction that $x_k \in S(x_0)$,
but $a_k = 0$ for some  $k \in \mathbb{N}$. Then from the optimality conditions \eqref{eq:HypodiffDist_OptCond} and 
the fact that the hypodifferential map $\underline{d} f(\cdot)$ is exact on $S_{\varepsilon}(x_0)$ it follows that
\begin{equation} \label{eq:MHD_Zero_a_k}
\begin{split}
  f(x_k - \alpha v_k) - f(x_k) = \Phi_k(- \alpha v_k)
  &= \max_{(a, v) \in \underline{d} f(x_k)} (a + \langle v, - \alpha v_k \rangle)
  \\
  &\le \alpha \max_{(a, v) \in \underline{d} f(x_k)} \langle v, - v_k \rangle \le - \alpha \| v_k \|^2
\end{split}
\end{equation}
for any $\alpha \ge 0$ such that $x_k(\alpha) := x_k - \alpha v_k \in S_{\varepsilon}(x_0)$. Let us show that 
$x_k(\alpha) \in S_{\varepsilon}(x_0)$ for any $\alpha \ge 0$. Then $f(x_k(\alpha)) \to - \infty$ as 
$\alpha \to + \infty$, which contradicts the fact that the function $f$ attains global minimum, since its sublevel set
$S(x_0)$ is bounded.

Indeed, note that $x_k(\alpha) \in S_{\varepsilon}(x_0)$ for any $\alpha \in [0, \varepsilon/ \| v_k \|]$. With the use
of the fact that the sublevel set $S(x_0)$ is convex one can easily check that the set $S_{\varepsilon}(x_0)$ is convex
as well. Therefore, either the ray $\{ x_k - \alpha v_k \mid \alpha \ge 0 \}$ is contained in $S_{\varepsilon}(x_0)$ or
there exists $\overline{\alpha} > 0$ such that
\[
  \big\{ x_k - \alpha v_k \bigm| \alpha \in [0, \overline{\alpha}] \} \subset S_{\varepsilon}(x_0), \quad
  x_k - \alpha v_k \notin S_{\varepsilon}(x_0) \quad \forall \alpha > \overline{\alpha}.
\]
Note, however, that in the latter case one has $f(x_k(\overline{\alpha})) < f(x_k) \le f(x_0)$ due to inequality
\eqref{eq:MHD_Zero_a_k}. Consequently, there exists $\delta > 0$ such that $f(x_k(\alpha)) \le f(x_0)$ for any
$\alpha \in [\overline{\alpha}, \overline{\alpha} + \delta]$ by virtue of the fact that $f$ is continuous by
\cite[Corollary~I.2.5]{EkelandTemam}. Thus, the second case is impossible and we can conclude that $a_k < 0$, whenever 
$x_k \in S(x_0)$.

\textbf{Part 2.} Let us show that $\{ x_k \} \subset S(x_0)$ and the sequence $\{ f(x_k) \}$ is strictly decreasing.
Indeed, let $x_k \in S(x_0)$ for some $k \in \mathbb{N}$. Then $a_k < 0$ and from the optimality conditions
\eqref{eq:HypodiffDist_OptCond} and the convexity of $\Phi_k$ it follows that
\[
  \Phi_k\left( \frac{\alpha}{a_k} v_k \right) \le \alpha \Phi_k\left( \frac{1}{a_k} v_k \right)
  \le \frac{\alpha}{a_k} \| (a_k, v_k) \|^2 < 0
\]
for any $\alpha \in [0, 1]$. Therefore applying the fact that the hypodifferential map $\underline{d} f(\cdot)$ is
exact on $S_{\varepsilon}(x_0)$ one obtains that
\[
  f\left(x_k + \frac{\alpha}{a_k} v_k \right) = f(x_k) + \Phi_k\left( \frac{\alpha}{a_k} v_k \right)
  \le f(x_k) + \frac{\alpha}{a_k} \| (a_k, v_k) \|^2
\]
for any $\alpha \in [0, 1]$ such that $x_k + (\alpha/a_k) v_k \in S_{\varepsilon}(x_k)$. Arguing in the same way as in
the first part of the proof one can show that this inclusion is satisfied for any $\alpha \in [0, 1]$. Consequently,
putting $\alpha = 1$ one gets
\begin{equation} \label{eq:MHD_a_k_StepSize_Decay}
  f(x_{k + 1}) \le f(x_k) + \frac{1}{a_k} \| (a_k, v_k) \|^2 < f(x_k).
\end{equation}
Therefore, $\{ x_k \} \subset S(x_0)$ and the sequence $\{ f(x_k) \}$ is strictly decreasing. 

\textbf{Part 3.} From inequality \eqref{eq:MHD_a_k_StepSize_Decay} it follows that
\begin{equation} \label{eq:DeltaDecay_a_k_StepSize}
  \Delta_{k + 1} - \Delta_k \le - \frac{1}{|a_k|} \| (a_k, v_k) \|^2
  \le - \frac{1}{\| (a_k, v_k) \|} \| (a_k, v_k) \|^2 = - \| (a_k, v_k) \|
\end{equation}
for any $k \in \mathbb{N}$. Hence applying inequality \eqref{eq:Delta_UpperEstim} one obtains
\[
  \Delta_{k + 1} \le \left( 1 - \frac{1}{1 + R} \right) \Delta_k \quad \forall k \in \mathbb{N},
\]
which implies that $\Delta_{k + 1} \le \Delta_0 q^{k + 1}$ for any $k \in \mathbb{N}$, where 
$q = 1 - \frac{1}{1 + R} \in (0, 1)$. Thus, the third statement of the theorem holds true. To prove the second
statement of the theorem, note that from inequality \eqref{eq:DeltaDecay_a_k_StepSize} it follows that
\[
  \dist(0, \underline{d} f(x_k)) = \| (a_k, v_k) \|
  \le \Delta_k - \Delta_{k + 1} \le \Delta_0 q^k - \Delta_0 q^{k + 1} = \Delta_0 (1 - q) q^k
\]
for any $k \in \mathbb{N}$. Thus, the second statement of the theorem holds true with $C = \Delta_0 (1 - q)$ and
$\gamma = q$.
\end{proof}

\begin{remark}
{(i)~If under the assumptions of the previous theorem the function $f$ is polyhedral, then by
\cite[Theorem~4.7]{Dolgopolik_HypodiffDescent} the method of hypodifferential descent with $\alpha_k = 1 / a_k$
converges to a point of global minimum of $f$ in a finite number of steps.
}

\noindent{(ii)~The previous theorem remains to hold true if the step sizes $\alpha_k$ are defined via the exact line
search:
\[
  f(x_k - \alpha_k v_k) = \min_{\alpha \ge 0} f(x_k - \alpha v_k).
\]
Indeed, in this case one obviously has
\[
  \Delta_{k + 1} - \Delta_k = \min_{\alpha \ge 0} f(x_k - \alpha v_k) - f(x_k)
  \le f\left( x_k + \frac{1}{a_k} v_k \right) - f(x_k) \le \frac{1}{a_k} \| (a_k, v_k) \|^2  
\]
(see \eqref{eq:MHD_a_k_StepSize_Decay}). Then arguing in the same way as in the third part of the proof of
Theorem~\ref{thm:MHD_ExactCase} one can prove the linear rate of convergence of the method.
}
\end{remark}

\subsection{Accelerated proximal hypodifferential descent}
\label{subsect:AcceleratedPHD}

Let us consider a different method for minimizing nonsmooth convex functions based on the use of hypodifferentials,
which we call \textit{the proximal hypodifferential descent method}. This method is related to various proximal point
methods \cite{Teboulle} and can be viewed as a generalisation of the Pschenichnyi-Pironneau-Polak algorithm for
convex minimax problems \cite[Algorithm~2.4.1]{Polak} (see also \cite[Section~2.3.3]{Nesterov_book}) to the case of
general nonsmooth convex functions, as well as a particular version of the quadratic regularisation of the method of
codifferential descent for nonconvex nonsmooth functions \cite{Dolgopolik_CodiffDescent}. Moreover, unlike the method of
hypodifferential descent from the previous section, the proximal hypodifferential descent can be applied to problems of
minimizing the function $f$ over a closed convex set $Q \subseteq X$. A theoretical scheme of this method is given in
Algorithm~\ref{alg:ProximalHypodiffDescent}.

\begin{algorithm} \label{alg:ProximalHypodiffDescent}
\caption{Proximal hypodifferential descent method.}

\textbf{Initialisation.} Choose an initial point $x_0 \in Q$, a paramter $\gamma > 0$, and put $k = 0$.

\textbf{Step 1.} Compute $\underline{d} f(x_k)$. Find an optimal solution $z_k$ of the problem
\begin{equation} \label{prob:ProximalHypodiff}
  \min_{x \in Q} \enspace \max_{(a, v) \in \underline{d} f(x_k)} (a + \langle v, x - x_k \rangle)
  + \frac{\gamma}{2} \| x - x_k \|^2.
\end{equation}

\textbf{Step 2.} Choose a step size $\alpha_k \in (0, 1]$ and put $x_{k + 1} = x_k + \alpha_k (z_k - x_k)$.

\textbf{Step 3.} Check a \textbf{stopping criterion}. If it is satisfied, \textbf{Stop}. Otherwise, 
put $k \leftarrow k + 1$ and go to \textbf{Step 1}.
\end{algorithm}

Let us comment on the proximal hypodifferential descent method. Firstly, note that the point $z_k$ is correctly defined
for any $k \in \mathbb{N}$. Indeed, it is easily seen that the function 
$\Psi_k(x) := \Phi_k(x - x_k) + 0.5 \gamma \| x - x_k \|^2$, $x \in X$, is strongly convex, where, as above, 
\[
  \Phi_k(x) = \max_{(a, v) \in \underline{d} f(x_k)} (a + \langle v, x \rangle) \quad \forall x \in X.
\]
Moreover, this function is coercive, since $|\Phi_k(x)| \le C_k + C_k \| x \|$ for all $x \in X$ and some $C_k > 0$ due
to the fact that hypodifferential is a bounded set. Therefore, the function $\Psi_k(\cdot)$ attain a global minimum on
$Q$ at a unique point $z_k$ (see, e.g. \cite[Proposition~II.1.2]{EkelandTemam}). 

If the set $Q$ is polyhedral (that is, it is defined by a finite number of linear equality and inequality
constraints) and the hypodifferential $\underline{d} f(x_k)$ is a polytope, then the auxiliary subproblem
\eqref{prob:ProximalHypodiff} on Step~1 of the proximal hypodifferential descent method can obviously be rewritten
as an equivalent convex quadratic programming problem. Let us also point out that the following inequalities
\begin{equation} \label{eq:PHD_StopCriterion}
  \| z_k - x_k \| \le \varepsilon \quad \text{and/or} \quad f(x_k) - f(x_{k + 1}) \le \varepsilon
\end{equation}
with small $\varepsilon > 0$ can be used as a stopping criterion for Algorithm~\ref{alg:ProximalHypodiffDescent}. 

The correctness of the stopping criteria \eqref{eq:PHD_StopCriterion}, as well as global convergence of the proximal
hypodifferential descent method can be proved in exactly the same way as the analogous results for QR-MCD method from
\cite[Section~4]{Dolgopolik_CodiffDescent}, since Algorithm~\ref{alg:ProximalHypodiffDescent} is a particular case of
QR-MCD in the convex case. In turn, the rate of convergence of the proximal hypodifferential descent method can be
estimated in the strongly convex case in exactly the same way as the rate of convergence of similar methods for convex
minimax problems in \cite[Theorem~2.3.4]{Nesterov_book} and \cite[Theorem~2.4.5]{Polak}. For the sake of shortness, we
do not present detailed formulations and proofs of these results here. Instead, we will analyse an accelerated version
of Algorithm~\ref{alg:ProximalHypodiffDescent} designed with the use of Nesterov's acceleration technique based on
estimating sequences \cite[Sections~2.2 and 2.3]{Nesterov_book}.

For any $x \in Q$ let $z_{\gamma}(x)$ be an optimal solution of the problem
\[
  \min_{z \in Q} \enspace \max_{(a, v) \in \underline{d} f(x)}(a + \langle v, z - x \rangle) 
  + \frac{\gamma}{2} \| z - x \|^2.
\]
Let us obtain one useful estimate.

\begin{lemma} \label{lem:ProximalLowerEstimate}
Let the hypodifferential map $\underline{d} f(\cdot)$ be consistent on $X$ and a Lipschitzian approximation of $f$ on
$X$ with Lipschitz constant $L > 0$. Then for any $\gamma \ge L$, $x \in Q$, and $y \in X$ one has 
\[
  f(x) \ge f(z_{\gamma}(y)) + \gamma \langle y - z_{\gamma}(y), x - y \rangle
  + \frac{\gamma}{2} \| y - z_{\gamma}(y) \|^2.
\]
\end{lemma}

\begin{proof}
Fix any $y \in X$ and denote
\[
  \Phi(x) = \max_{(a, v) \in \underline{d} f(y)}(a + \langle v, x - y \rangle), \quad
  \Psi(x) = \Phi(x) + \frac{\gamma}{2} \| x - y \|^2 \quad \forall x \in X.
\]
Clearly, the function $\Psi$ is strongly convex with modulus $\mu = \gamma$. Therefore
\[
  \Psi(x) \ge \Psi(z_{\gamma}(y)) + \frac{\gamma}{2} \| x - z_{\gamma}(y) \|^2 \quad \forall x \in Q,
\]
(see, e.g. \cite[Proposition~4.8]{Vial}). Hence for any $x \in Q$ one has
\begin{align*}
  \Phi(x) &= \Psi(x) - \frac{\gamma}{2} \| x - y \|^2
  \ge \Psi(z_{\gamma}(y)) + \frac{\gamma}{2} \Big( \| x - z_{\gamma}(y) \|^2 - \| x - y \|^2 \Big)
  \\
  &= \Psi(z_{\gamma}(y)) + \frac{\gamma}{2} \langle y - z_{\gamma}(y), 2x - z_{\gamma}(y) - y \rangle
  \\
  &= \Psi(z_{\gamma}(y)) + \frac{\gamma}{2} \langle y - z_{\gamma}(y), 2 (x - y) + y - z_{\gamma}(y) \rangle 
  \\
  &= \Psi(z_{\gamma}(y)) + \gamma \langle y - z_{\gamma}(y), x - y \rangle 
  + \frac{\gamma}{2} \| y - z_{\gamma}(y) \|^2.
\end{align*}
From the facts that $\underline{d} f(\cdot)$ is a Lipschitzian approximation of $f$ on $X$ with Lipschitz constant 
$L > 0$ and $\gamma \ge L$ it follows that
\[
  f(z_{\gamma}(y)) \le f(y) + \Phi(z_{\gamma}(y)) + \frac{L}{2} \| z_{\gamma}(y) - y \|^2
  = f(y) + \Psi(z_{\gamma}(y)). 
\]
Hence taking into account the fact that the hypodifferential map $\underline{d} f(\cdot)$ is consistent one obtains
\begin{align*}
  f(x) &\ge f(y) + \Phi(x) 
  \\
  &\ge f(y) + \Psi(z_{\gamma}(y)) + \gamma \langle y - z_{\gamma}(y), x - y \rangle 
  + \frac{\gamma}{2} \| y - z_{\gamma}(y) \|^2
  \\
  &\ge f(z_{\gamma}(y)) + \gamma \langle y - z_{\gamma}(y), x - y \rangle 
  + \frac{\gamma}{2} \| y - z_{\gamma}(y) \|^2
\end{align*}
for any $x \in Q$.
\end{proof}

Suppose that the hypodifferential map $\underline{d} f(\cdot)$ is a Lipschitzian approximation of $f$ on $X$ with
Lipschitz constant $L > 0$. Following the derivation of Nesterov's accelerated method for convex minimax problems
\cite[Section~2.3.3]{Nesterov_book}, choose some $x_0 \in Q$ and $\gamma_0 > 0$. Let $\{ y_k \} \subset Q$ and 
$\{ \alpha_k \} \subset (0, 1)$ be some sequences. For any $k \in \mathbb{N}$ define
\begin{align*}
  \phi_0(x) &= f(x_0) + \frac{\gamma_0}{2} \| x - x_0 \|^2,
  \\
  \phi_{k + 1}(x) &= \begin{aligned}[t]
    &(1 - \alpha_k) \phi_k(x) + \alpha_k \left( f(z_L(y_k)) + \frac{L}{2} \| y_k - z_L(y_k) \|^2 \right)
    \\
    &+ \alpha_k L \langle y_k - z_L(y_k), x - y_k \rangle.
  \end{aligned}
\end{align*}
Since $\phi_0$ is a quadratic function and $\phi_{k + 1}$ is the sum of $(1 - \alpha_k) \phi_k(x)$ and an affine
function, the functions $\phi_k$ are quadratic for all $k \in \mathbb{N}$. Moreover, one can readily check that
\[
  \phi_k(x) = \phi_k^* + \frac{\gamma_k}{2} \| x - v_k \|^2 \quad \forall k \in \mathbb{N},
\]
where 
\begin{align} \label{eq:APHD_gamma_k}
  v_0 &= x_0, \quad \phi_0^* = f(x_0), \quad \gamma_{k + 1} = (1 - \alpha_k) \gamma_k,
  \\ \label{eq:APHD_v_k}
  v_{k + 1} &= v_k - \frac{\alpha_k L}{\gamma_{k + 1}} (y_k - z_L(y_k)),
  \\ \label{eq:APHD_phi_k_min}
  \phi_{k + 1}^* &= \begin{aligned}[t]
    &(1 - \alpha_k) \phi_k^* + \alpha_k \left( f(z_L(y_k)) + \frac{L}{2} \| y_k - z_L(y_k) \|^2 \right)
    \\
    &- \frac{\alpha_k^2 L^2}{2 \gamma_{k + 1}} \| y_k - z_L(y_k) \|^2 
    + \alpha_k L \langle y_k - z_L(y_k), v_k - y_k \rangle
  \end{aligned}
\end{align}
for any $k \in \mathbb{N}$. With the use of the above expression for the functions $\phi_k$ we can derive an accelerated
version of the proximal hypodifferential descent method.

\begin{lemma} \label{lem:MethodDerivation}
Let $\alpha_0 \in (0, 1)$, $\gamma_0 = L \alpha_0^2$ and sequences $\{ x_k \}$, $\{ y_k \}$, and 
$\{ \alpha_k \}_{k \ge 1}$ be chosen according to the following equalities:
\begin{equation} \label{eq:OptimalSteps}
  x_{k + 1} = z_L(y_k), \quad L \alpha_k^2 = \gamma_{k + 1}, \quad 
  y_k = \alpha_k v_k + (1 - \alpha_k) x_k.
\end{equation}
Then for all $k \in \mathbb{N}$ one has 
\[
  \alpha_{k + 1}^2 = (1 - \alpha_{k + 1}) \alpha_k^2, \quad \alpha_k \in (0, 1), \quad \phi_k^* \ge f(x_k).
\]
\end{lemma}

\begin{proof}
The equality for $\alpha_{k + 1}$ follows directy from equalities $L \alpha_k^2 = \gamma_{k + 1}$ and 
\eqref{eq:APHD_gamma_k}. Note also that the equation $t^2 = (1 - t) \alpha_k^2$ has a unique positive solution
that lies in the interval $(0, 1)$, since the function $g(t) = t^2 - \alpha_k^2 (1 - t)$ is strictly increasing on 
$[0, 1]$ and satisfies the equalities $g(0) = - \alpha_k^2 < 0$ and $g(1) = 1$. Therefore, $\alpha_k \in (0, 1)$ for all
$k \in \mathbb{N}$.

Let us prove the inequality $\phi_k^* \ge f(x_k)$ by induction in $k$. Suppose that $\phi_k^* \ge f(x_k)$ for some 
$k \in \mathbb{N}$ (recall that $\phi_0^* = f(x_0)$). Applying Lemma~\ref{lem:ProximalLowerEstimate} with $x = x_k$,
$y = y_k$, and $\gamma = L$ one gets
\[
  f(x_k) \ge f(z_L(y_k)) + L \langle y_k - z_L(y_k), x_k - y_k \rangle + \frac{L}{2} \| y_k - z_L(y_k)) \|^2. 
\]
Hence with the use of \eqref{eq:APHD_phi_k_min} and inequality $\phi_k^* \ge f(x_k)$ one obtains
\begin{align*}
  \phi_{k + 1}^* 
  &\ge \begin{aligned}[t] 
    &(1 - \alpha_k) f(x_k) + \alpha_k \left( f(z_L(y_k)) + \frac{L}{2} \| y_k - z_L(y_k) \|^2 \right)
    \\
    &- \frac{\alpha_k^2 L^2}{2 \gamma_{k + 1}} \| y_k - z_L(y_k) \|^2
    + \alpha_k L \langle y_k - z_L(y_k), v_k - y_k \rangle
    \end{aligned}
  \\
  &\ge \begin{aligned}[t]
    &f(z_L(y_k)) + \left( \frac{L}{2} - \frac{\alpha_k^2 L^2}{2 \gamma_{k + 1}} \right) \| y_k - z_L(y_k) \|^2
    \\
    &+ (1 - \alpha_k) 
    L \left\langle y_k - z_L(y_k), x_k - y_k + \frac{\alpha_k}{(1 - \alpha_k)} (v_k - y_k) \right\rangle
    \end{aligned}
\end{align*}
Consequently, $\phi_{k + 1}^* \ge f(x_{k + 1})$ due to equalities \eqref{eq:OptimalSteps}.
\end{proof}

Observe that from \eqref{eq:OptimalSteps} and \eqref{eq:APHD_v_k} it follows that
\begin{align*}
  v_{k + 1} &= v_k - \frac{\alpha_k L}{\gamma_{k + 1}} (y_k - z_L(y_k))
  = v_k - \frac{1}{\alpha_k} (y_k - x_{k + 1})
  \\
  &= - \frac{1}{\alpha_k} (y_k - \alpha_k v_k) + \frac{1}{\alpha_k} x_{k + 1}
  = - \frac{1}{\alpha_k} (1 - \alpha_k) x_k + \frac{1}{\alpha_k} x_{k + 1}
  \\
  &= x_k + \frac{1}{\alpha_k} (x_{k + 1} - x_k).
\end{align*}
Hence with the use of \eqref{eq:OptimalSteps} one gets
\begin{align*}
  y_{k + 1} &= \alpha_{k + 1} v_{k + 1} + (1 - \alpha_{k + 1}) x_{k + 1}
  \\
  &= \alpha_{k + 1} x_k + \frac{\alpha_{k + 1}}{\alpha_k} (x_{k + 1} - x_k) + (1 - \alpha_{k + 1}) x_{k + 1}
  \\
  &= x_{k + 1} + \frac{\alpha_{k + 1}(1 - \alpha_k)}{\alpha_k} (x_{k + 1} - x_k).
\end{align*}
Thus, we arrive at the following method, which we call \textit{accelerated proximal hypodifferential descent method},
whose theoretical scheme is given in Algorithm~\ref{alg:APHD}. 

\begin{algorithm} \label{alg:APHD}
\caption{Accelerated proximal hypodifferential descent method.}

\textbf{Initialisation.} Choose an initial point $x_0 \in Q$ and parameter $\alpha_0 \in (0, 1)$. Set $y_0 = x_0$ and
$k = 0$.

\textbf{Step 1.} Compute $\underline{d} f(y_k)$. Find an optimal solution $x_{k + 1}$ of the problem
\[
  \min_{x \in Q} \enspace \max_{(a, v) \in \underline{d} f(y_k)} (a + \langle v, x - y_k \rangle)
  + \frac{L}{2} \| x - y_k \|^2.
\]

\textbf{Step 2.} Find a solution $\alpha_{k + 1} \in (0, 1)$ of the equation
\[
  \alpha_{k + 1}^2 = (1 - \alpha_{k + 1}) \alpha_k^2
\]
and put
\[
  y_{k + 1} = x_{k + 1} + \frac{\alpha_{k + 1}(1 - \alpha_k)}{\alpha_k} (x_{k + 1} - x_k).
\]
Set $k \leftarrow k + 1$ and go to \textbf{Step 1}.
\end{algorithm}

The following theorem describes the rate of convergence of this method. Let us note that its proof largely repeats the
way in which the rate of convergence of Nesterov's accelerated gradient descent method is estimated
\cite{Nesterov_book}.

\begin{theorem}
Suppose that $f$ attains a global minimum on the set $Q$ and the hypodifferential map $\underline{d} f(\cdot)$ is
consistent on $X$ and a Lipschitzian approximation of $f$ on $X$ with Lipschitz constant $L > 0$. Let $\{ x_k \}$ be 
the sequence generated by the accelerated proximal hypodifferential descent method. Then for any $k \in \mathbb{N}$ one
has
\[
  f(x_k) - f_* 
  \le \frac{4L}{(2 \sqrt{L} + k \sqrt{\gamma_0})^2} \Big( f(x_0) - f_* + \frac{\gamma_0}{2} \| x_0 - x_* \|^2 \Big)
  = \mathcal{O}\left( \frac{1}{k^2} \right),
\]
where $f_* = \min_{x \in Q} f(x_*)$ and $x_*$ is a point of global minimum of $f$ on $Q$.
\end{theorem}

\begin{proof}
Define $\lambda_0 = 1$ and $\lambda_{k + 1} = (1 - \alpha_k) \lambda_k$ for all $k \in \mathbb{N}$. Let us check that
\begin{equation} \label{eq:EstimatingSeq}
  \phi_k(x) \le (1 - \lambda_k) f(x) + \lambda_k \phi_0(x) \quad \forall x \in X, \: k \in \mathbb{N}.
\end{equation}
Indeed, in the case $k = 0$ the inequality is obvious. Suppose that it is satisfied for some $k \in \mathbb{N}$. 
Recall that 
\begin{align*}
  \phi_{k + 1}(x) 
  = &(1 - \alpha_k) \phi_k(x) + \alpha_k \left( f(z_L(y_k)) + \frac{L}{2} \| y_k - z_L(y_k) \|^2 \right)
  \\
  &+ \alpha_k L \langle y_k - z_L(y_k), x - y_k \rangle.
\end{align*}
Applying Lemma~\ref{lem:ProximalLowerEstimate} with $y = y_k$ one gets that
\[
  \phi_{k + 1}(x) \le (1 - \alpha_k) \phi_k(x) + \alpha_k f(x) \quad \forall x \in X.
\]
Hence with the use of inequality \eqref{eq:EstimatingSeq} one obtains that
\begin{align*}
  \phi_{k + 1}(x) &\le (1 - \alpha_k) \Big( (1 - \lambda_k) f(x) + \lambda_k \phi_0(x) \Big) + \alpha_k f(x)
  \\
  &= (1 - \lambda_{k + 1}) f(x) + \lambda_{k + 1} \phi_0(x)
\end{align*}
for any $x \in X$. Thus, by induction inequality \eqref{eq:EstimatingSeq} holds true for all $k \in \mathbb{N}$.

By the construction of the method one has $f(x_k) \le \phi_k^*$. Therefore by applying inequality
\eqref{eq:EstimatingSeq} one obtains that
\begin{align*}
  f(x_k) \le \phi_k^* = \min_{x \in X} \phi_k(x) \le \min_{x \in Q} \phi_k(x)
  &\le \inf_{x \in Q} \Big( (1 - \lambda_k) f(x) + \lambda_k \phi_0(x) \Big) 
  \\
  &\le (1 - \lambda_k) f(x_*) + \lambda_k \phi_0(x_*)
\end{align*}
or, equivalently,
\begin{equation} \label{eq:EstimSeqIneq}
  f(x_k) - f_* \le \lambda_k (\phi_0(x_*) - f_*)
  = \lambda_k \left[ f(x_0) - f_* + \frac{\gamma_0}{2} \| x_0 - x_* \|^2 \right]
\end{equation}
for any $k \in \mathbb{N}$.

Let us estimate $\lambda_k$. First, note that $\gamma_k \ge \gamma_0 \lambda_k$ for all $k \in \mathbb{N}$. Indeed, 
$\gamma_0 = \gamma_0 \lambda_0$, since $\lambda_0 = 1$ by definition. Suppose that the inequality holds true for some
$k \in \mathbb{N}$. Then
\[
  \gamma_{k + 1} = (1 - \alpha_k) \gamma_k \ge (1 - \alpha_k) \gamma_0 \lambda_k = \gamma_0 \lambda_{k + 1},
\]
which by induction implies that $\gamma_k \ge \gamma_0 \lambda_k$ or, equivalently, 
$\lambda_k \le \gamma_k / \gamma_0$ for all $k \in \mathbb{N}$.

For any $k \in \mathbb{N}$ one has
\[
  \frac{1}{\sqrt{\lambda_{k + 1}}} - \frac{1}{\sqrt{\lambda_k}} 
  = \frac{\sqrt{\lambda_k} - \sqrt{\lambda_{k + 1}}}{\sqrt{\lambda_k \lambda_{k + 1}}}
  = \frac{\lambda_k - \lambda_{k + 1}}{(\sqrt{\lambda_k} + \sqrt{\lambda_{k + 1}})\sqrt{\lambda_k \lambda_{k + 1}}}.
\]
Recall that $\lambda_{k + 1} = (1 - \alpha_k) \lambda_k$ and $\alpha_k \in (0, 1)$. Therefore
\begin{align*}
  \frac{1}{\sqrt{\lambda_{k + 1}}} - \frac{1}{\sqrt{\lambda_k}}
  &= \frac{\lambda_k - (1 - \alpha_k) \lambda_k}
  {(\sqrt{\lambda_k} + \sqrt{\lambda_{k + 1}})\sqrt{\lambda_k \lambda_{k+ 1}}}
  \\
  &= \frac{\alpha_k}{(1 + \sqrt{1- \alpha_k}) \sqrt{\lambda_{k + 1}}}
  \ge \frac{\alpha_k}{2 \sqrt{\lambda_{k + 1}}}.
\end{align*}
Hence applying the inequality $\lambda_{k + 1} \le \gamma_{k + 1} / \gamma_0$ and the equality 
$L \alpha_k^2 = \gamma_{k + 1}$ (see Lemma~\ref{lem:MethodDerivation}) one obtains that
\[
  \frac{1}{\sqrt{\lambda_{k + 1}}} - \frac{1}{\sqrt{\lambda_k}}
  \ge \frac{\alpha_k}{2} \sqrt{\frac{\gamma_0}{\gamma_{k + 1}}} = \frac{1}{2} \sqrt{\frac{\gamma_0}{L}}
\]
for any $k \in \mathbb{N}$, which implies that
\[
  \frac{1}{\sqrt{\lambda_k}} \ge 1 + \frac{k}{2} \sqrt{\frac{\gamma_0}{L}} \quad
  \Longleftrightarrow \quad \lambda_k \le \frac{4L}{(2 \sqrt{L} + k \sqrt{\gamma_0})^2}
  \quad \forall k \in \mathbb{N}.
\]
Combining this inequality with \eqref{eq:EstimSeqIneq} one obtains the required result.
\end{proof}

\section{Conclusions}

We presented a general theory of hypodifferentials of nonsmooth convex functions and its applications to nonsmooth
convex optimization. The key difference between subdifferentials and hypodifferentials is the fact that, in a sense,
hypodifferentials much closer resemble the gradient of a smooth convex function in its properties than
subdifferentials. In particular, hypodifferentials possess many important properties of the gradient that are widely
used in convex optimization (such as continuity and Lipschitz continuity, the property of being a Lipschitzian
approximation, etc.), that subdifferentials do not have. This important feature of hypodifferentials allowed us to
extend some of the results on smooth convex functions/convex optimization methods to the general nonsmooth setting,
which is apparently impossible to do in the context of the theory of subdifferentials. For example, as we have shown,
hypodifferentials can be used to generalize accelerated gradient methods to the case of nonsmooth convex functions. 
Thus, from the theoretical point of view, the theory of hypodifferentials provides one with new useful tools for
studying nonsmooth convex functions and problems.

Although the calculus rules and examples presented in the paper show how hypodifferentials of nonsmooth convex
functions can be computed in particular cases, it must be said that hypodifferentials are often much more cumbersome and
hard to deal with from the computational point of view than subdifferentials/subgradients. They provide significantly
more information about the function than subdifferentials/subgradients at the expense of much higher computational cost
and memory consumption. In particular, the higher rate of convergence of convex optimization methods based on
hypodifferentials in comparison with standard nonsmooth convex optimization methods is achieved by a significant
increase of complexity of each iteration and each oracle call, which makes them inapplicable to large scale problems.
However, in some applications (especially in the case of problems of moderate dimensions), the higher rate of
convergence of accelerated hypodifferential descent methods might compensate for the significantly higher complexity of
each iteration to make such methods competitive in comparison with classical subgradient/bundle methods.

\bibliographystyle{abbrv}  
\bibliography{Hypodiff_bibl}

\end{document}